\documentclass[12pt,twoside]{article}          
\usepackage[latin1]{inputenc}
\usepackage[T1]{fontenc}
\usepackage[english]{babel}
\usepackage{amsmath, amssymb, amsthm}            
\usepackage{amstext, amsfonts, a4}
\usepackage{tikz}
\usepackage{pgf,tikz}
\usepackage{mathrsfs}
\usetikzlibrary{arrows}

 \usepackage[colorlinks=true]{hyperref}
\hypersetup{urlcolor=blue, citecolor=blue}

\usepackage{color,transparent}
\usepackage[lofdepth,lotdepth]{subfig}

\usepackage{comment}
\usepackage{float} 
\theoremstyle{plain}
	\newtheorem{Theo}{Theorem}[section] 
	\newtheorem{Prop}[Theo]{Proposition}        
	\newtheorem{Lemm}[Theo]{Lemma}            
	\newtheorem{Coro}[Theo]{Corollaire}

	\newtheorem{Rema}[Theo]{Remark}

\theoremstyle{definition}
	\newtheorem{Defi}[Theo]{Definition}

\theoremstyle{remark}

\def\ogg~{{\rm \og}}   

\def\emptyset{\varnothing}

\def\RR{{\mathbb R}}    
\def\PP{{\mathbb P}}     
\def\EE{{\mathbb E}}    

\def\dt{{\mathrm d}}

\def\Br{{\mathcal{B}_r}}
\def\S{{\mathbb{S}^1}}
\def\SS{{\mathbb{S}}}

\definecolor{qqwuqq}{rgb}{0.,0.39215686274509803,0.}
\definecolor{xdxdff}{rgb}{0.49019607843137253,0.49019607843137253,1.}
\definecolor{ffqqqq}{rgb}{1.,0.,0.}
\definecolor{uuuuuu}{rgb}{0.26666666666666666,0.26666666666666666,0.26666666666666666}

\makeatletter
\def\BState{\State\hskip-\ALG@thistlm}
\makeatother

\begin{document}
\title{Explicit speed of convergence of the stochastic billiard in a convex set}

\author{Ninon F\'etique \footnotemark[1]}
\date{ }

\footnotetext[1]{Laboratoire de Math\'ematiques et Physique Th\'eorique (UMR CNRS 7350), F\'ed\' eration Denis Poisson (FR CNRS
2964), Universit\'e Fran\c cois-Rabelais, Parc de Grandmont, 37200 Tours, France. Email: ninon.fetique@lmpt.univ-tours.fr}

\maketitle

\begin{abstract}
In this paper, we are interested in the speed of convergence of the stochastic billiard evolving in a convex set $K$. This process can be described as follows: a particle moves at unit speed inside the set $K$ until it hits the boundary, and is randomly reflected, independently of its position and previous velocity. We focus on convex sets in $\RR^2$ with a curvature bounded from above and below. We give an explicit coupling for both the continuous-time process and the embedded Markov chain of hitting points on the boundary, which leads to an explicit speed of convergence to equilibrium.
\end{abstract}

\begin{flushleft}
\textbf{MSC Classification 2010:} 60J05, 60J25, 60J75, 60F17 . \\ 
\textbf{Key words:} Stochastic billiard, invariant measure, coupling, speed of convergence.
\end{flushleft}

\tableofcontents

\bigskip

\section{Introduction}

In this paper, our goal is to give explicit bounds on the speed of convergence of a process, called "stochastic billiard", towards its invariant measure, under some assumptions that we will detail further. This process can be informally described as follows: a particle moves at unit speed inside a domain until it hits the boundary. At this time, the particle is reflected inside the domain according to a random distribution on the unit sphere, independently on its position and previous velocity.\\
The stochastic billiard is a generalisation of shake-and-bake algorithm (see \cite{Romeijn}), in which the reflection law is the cosine law. In that case, it has been proved that the Markov chain of hitting points on the boundary has a uniform stationary distribution. In \cite{Romeijn}, the shake-and-bake algorithm is introduced for generating uniform points on the boundary of bounded polyhedra. More generally, stochastic billiards can be used for sampling from a bounded set or the boundary of such a set, through the Markov Chain Monte Carlo algorithms. In that sense, it is therefore important to have an idea of the speed of convergence of the process towards its invariant distribution.\medskip\\
Stochastic billiards have been studied a lot, under different assumptions on the domain in which it lives and on the reflection law. Let us mention some of these works. In \cite{Evans}, Evans considers the stochastic billiard with uniform reflection law in a bounded $d$-dimensional region with $C^1$ boundary, and also in polygonal regions in the plane. He proves first the exponentially fast total variation convergence of the Markov chain, and moreover the uniform total variation C\'esaro convergence for the continuous-time process. In \cite{DV}, the authors only consider the stochastic billiard Markov chain, in a bounded convex set with curvature bounded from above and with a cosine distribution for the reflection law. They give a bound for the speed of convergence of this chain towards its invariant measure, that is the uniform distribution on the boundary of the set, in order to get a bound for the number of steps of the Markov chain required to sample approximatively the uniform distribution. Finally, let us mention the work of Comets, Popov, Sch\"utz and Vachkovskaia \cite{CPSV}, in which some ideas have been picked and used in the present paper. They study the convergence of the stochastic billiard and its associated Markov chain in a bounded domain in $\RR^d$ with a boundary locally Lipschitz and almost everywhere $C^1$. They consider the case of a reflection law which is absolutely continuous with respect to the Haar measure on the unit sphere of $\RR^d$, and supported on the whole half-sphere that points into the domain. They show the exponential ergodicity of the Markov chain and the continuous-time process and also their Gaussian fluctuations. The particular case of the cosine reflection law is discussed. Even if they do not give speeds of convergence, their proofs could lead to explicit speeds if we write them in particular cases (as for the stochastic billiard in a disc of $\RR^2$ for instance). However, as we will mention in Section \ref{Subsection A coupling for the stochastic billiard}, the speed of convergence obtained in particular cases will not be relevant, since their proof is adapted to their very general framework, and not for more particular and simple domains. \medskip\\
The goal of this paper is to give explicit bounds on the speed of convergence of the stochastic billiard and its embedded Markov chain towards their invariant measures. For that purpose, we are going to give an explicit coupling of which we can estimate the coupling time.\\
In a first part, we study the particular case of the billiard in a disc. In that case, everything is quite simple since all the quantities can be explicitly expressed. \\
Then, in a second part, we extend the results for the case of the stochastic billiard in a compact convex set of $\RR^2$ with curvature bounded from above and below. In that case, we can no more do explicit computations on the quantities describing the process, since we do not know exactly the geometry of the convex set. However, thanks to the assumptions on the curvature, we are able to estimate the needed quantities.\\
In both cases, the disc and the convex set, we suppose that the reflection law has a density function which is bounded from below by a strictly positive constant on a part of the sphere. The speed of convergence will obviously depend on it. However, for the convergence of the stochastic billiard process in a convex set, we will need to suppose that the reflection law is supported on the whole half sphere that points inside the domain.\\
At the end of this paper, we briefly discuss the extension of the results to higher dimensions. 

\subsection*{Notations}

We introduce some notations used in the paper:
\begin{itemize}
\item for $A\subset \RR$, $\mathbf{1}_A$ denotes the indicator function of the set $A$, that is $\mathbf{1}_A(x)$ is equal to $1$ if $x\in A$ and $0$ otherwise;
\item for $x\in \RR$, $\lfloor x \rfloor$ denotes the floor of the real $x$;
\item for $x,y\in\RR^2$, we note by $\lVert x \lVert$ the euclidean norm of $x$ and we write $\langle x,y \rangle$ for the scalar product  of $x$ and $y$;
\item for $A\subset\RR^2$, $\partial A$ denotes the boundary of the set $A$;
\item $\Br$ denotes the closed ball of $\RR^2$ centred at the origin with radius $r$, i.e. $\Br=\left\{ x\in \RR^2: \lVert x \lVert \leq r \right\}$, and $\S$ denotes the unit sphere of $\RR^2$, i.e. $\S=\left\{ x\in\RR^2: \lVert x \lVert=1 \right\}$;
\item for $\mathcal{I}\subset \RR$, $\lvert \mathcal{I}\lvert$ denotes the Lebesgue measure of the set $\mathcal{I}$;
\item for $K\subset \RR^2$ a compact convex set, we consider the 1-dimensional Hausdorff measure in $\RR^2$ restricted to $\partial K$. Therefore, if $A\subset \partial K$, $\lvert A \lvert$ denotes this Hausdorff measure of $A$; 
\item for $A\subset \RR^2$, if $x\in \partial A$, we write $n_x$ the unitary normal vector of $\partial A$ at $x$ looking into $A$ and we define $\SS_x$ the set of vectors that point the interior of $A$: $\SS_x=\left\{v\in\S: \langle v, n_x\rangle \geq 0  \right\}$;
\item if two random variables $X$ and $Y$ are equal in law we write $X\overset{\mathcal{L}}{=}Y$, and we write $X\sim\mu$ to say that the random variable $X$ has $\mu$ for law;
\item we denote by $\mathcal{G}(p)$ the geometric law with parameter $p$.
\end{itemize}

\section{Coupling for the stochastic billiard}

\subsection{Generalities on coupling}\label{Subsection Generalities on coupling}

In order to describe the way we will prove the exponential convergences and obtain bounds on the speeds of convergence, we first need to introduce some notions.\\
Let $\nu$ and $\overset{\sim}{\nu}$ be two probability measures on a measurable space $E$. We say that a probability measure on $E\times E$ is a coupling of $\nu$ and $\overset{\sim}{\nu}$ if its marginals are $\nu$ and $\overset{\sim}{\nu}$. Denoting by $\Gamma(\nu,\overset{\sim}{\nu})$ the set of all the couplings of $\nu$ and $\overset{\sim}{\nu}$, we say that two random variables $Y$ and $\overset{\sim}{Y}$ satisfy $(Y,\overset{\sim}{Y})\in\Gamma(\nu,\overset{\sim}{\nu})$ if $\nu$ and $\overset{\sim}{\nu}$ are the respective laws of $Y$ and $\overset{\sim}{Y}$. The total variation distance between these two probability measures is then defined by 
\begin{equation*}
\lVert \nu - \overset{\sim}{\nu} \lVert_{TV} = \underset{(Y,\overset{\sim}{Y})\in\Gamma(\nu,\overset{\sim}{\nu})}{\inf}  \PP(Y\neq \overset{\sim}{Y}).
\end{equation*}
For other equivalent definitions of the total variation distance and its properties, see for instance \cite{Lind}.\\
Let $(Y)_{t\geq 0}$ and $(\tilde{Y})_{t\geq 0}$ be two Markov processes and let define $T_c=\inf\left\{t\geq 0: Y_t=\tilde{Y}_t\right\} $ the coupling time of $Y$ and $\tilde{Y}$. From the definition of the total variation distance, it immediately follows that
\begin{equation*}
\lVert \mathcal{L}(Y_t)-\mathcal{L}(\tilde{Y}_t)  \lVert_{TV} \leq  \PP\left( T_c>t \right).
\end{equation*}
Therefore, let $T^*$ be a random variable stochastically bigger than $T_c$, $T_c\leq_{st} T^*$, which means that $\PP\left( T_c\leq t \right) \geq \PP\left( T^*\leq t \right)$ for all $t\geq 0$. If $T^*$ has a finite exponential moment, Markov's inequality gives then, for any $\lambda$ such that the Laplace transform of $T^*$ is well defined:
\begin{equation*}
\lVert \mathcal{L}(Y_t) - \mathcal{L}(\tilde{Y}_t) \lVert_{TV} \leq \PP\left( T^* >t \right) \leq \mathrm{e}^{-\lambda t}\EE\left[ \mathrm{e}^{\lambda T^*} \right].
\end{equation*}
Thus, if we manage to stochastically bound the coupling time of two stochastic billiards by a random time whose Laplace transform can be estimated, we get an exponential bound for the speed of convergence of the stochastic billiard towards its invariant measure.\medskip\\
We end this part with a definition that we will use throughout this paper.
\begin{Defi}
Let $K\subset \RR^2$ be a compact convex set.\\
We say that a pair of random variables $(X,T)$ living in $\partial K\times \RR^+$ is $\alpha$-continuous on the set $A\times B\subset \partial K\times \RR^+$ if for any measurable $A_1\subset A$, $B_1\subset B$:
\[ \PP\left( X\in A_1, T\in B_1\right) \geq \alpha |A_1| |B_1|. \]
\end{Defi}
We can also adapt this definition for a single random variable.

\subsection{Description of the process}\label{Subsection Description of the process}

Let us now give a precise description of the stochastic billiard $(X_t,V_t)_{t\geq 0}$ is a set $K$.\\
We assume that $K\subset \RR^2$ is a compact convex set with a boundary at least $C^1$.\\
Let $e=(1,0)$ be the first coordinate vector of the canonical base of $\RR^2$. We consider a law $\gamma$ on the half-sphere $\SS_e=\{v\in \mathbb{S}^1: e\cdot v\geq 0  \}$. Let moreover $(U_x, x\in \partial K)$ be a family of rotations of $\S$ such that $U_xe=-n_x$, where  we recall that $n_x$ is the normal vector of $\partial K$ at $x$ looking into $K$.\\
Given $x_0\in \partial K$, we consider the process $(X_t,V_t)_{t\geq 0}$ living in $K\times \mathbb{S}^1$ constructed as follows (see Figure \ref{Figure billard}):

\begin{itemize}
 \item Let $X_0=x_0$, and $V_0=U_{X_0}\eta_0$ with $\eta_0$ a random vector chosen according to the distribution $\gamma$.
\item Let $\tau_1=\inf\{t>0: x_0+tV_0\notin K\} $ and define $T_1=\tau_1$. We put $X_t=x_0+tV_0$, $V_t=V_0$ for $t\in [0,T_1)$, and $X_{T_1}=x_0+\tau_1V_0$.\\
Then, let $V_{T_1}=U_{X_{T_1}}\eta_1$ with $\eta_1$ following the law $\gamma$.
\item Let $\tau_2=\inf\{t>0: X_{T_1}+tV_{T_1}\notin K\} $ and define $T_2=T_1+\tau_2$. We put $X_t=X_{T_1}+tV_{T_1}$, $V_t=V_{T_1}$ for $t\in [T_1,T_2)$, and $X_{T_2}=X_{T_1}+\tau_2V_{T_1}$.\\
Then, let $V_{T_2}=U_{X_{T_2}}\eta_2$ with $\eta_2$ following the law $\gamma$.\\
\item And we start again ...
 \end{itemize}
 As mentioned in the introduction $(X_{T_n})_{n\geq 0}$ is a Markov chain living in $\partial K$ and the process $(X_t,V_t)_{t\geq 0}$ is a Markov process living in $K\times \S$.
 
 \definecolor{xdxdff}{rgb}{0.49019607843137253,0.49019607843137253,1.}
\definecolor{ududff}{rgb}{0.30196078431372547,0.30196078431372547,1.}
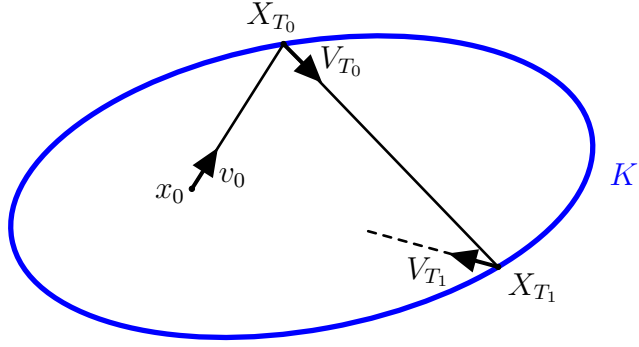
\begin{figure}
\begin{center}
\begin{tikzpicture}[line cap=round,line join=round,>=triangle 45,x=1.0cm,y=1.0cm,scale=1]
\clip(-5,-3) rectangle (4,3.5);

\draw [color=blue,rotate around={-169.6879509628261:(-0.8813058802500431,0.33123032309989775)},line width=2.pt] (-0.8813058802500431,0.33123032309989775) ellipse (3.9170357228136985cm and 1.9101889339569433cm);
\draw [line width=1pt] (-2.3420067194346355,0.30227099552412634)-- (-1.12317761651767,2.228119049044203);
\draw [->,line width=1.5pt] (-2.3420067194346355,0.30227099552412634) -- (-1.995581499186186,0.8496507180045398);
\draw [line width=1.pt] (-1.12317761651767,2.228119049044203)-- (1.7298151282607106,-0.7365751697717041);
\draw [->,line width=1.5pt] (-1.12317761651767,2.228119049044203) -- (-0.609679728621853,1.7123045926973206);
\draw [->,line width=1.5pt] (1.7298151282607106,-0.7365751697717041) -- (1.0689316965843305,-0.5571780541814397);
\draw [line width=1.pt,dashed] (1.7298151282607106,-0.7365751697717041)-- (-0.04566628975257539,-0.2483135519435019);

\draw[color=blue] (3.4,0.5) node {$K$};
\draw [fill=black] (-2.3420067194346355,0.30227099552412634) circle (1pt);
\draw[color=black] (-2.6307278845700983,0.235126538515879) node {$x_0$};
\draw [fill=black] (-1.12317761651767,2.228119049044203) circle (1pt);
\draw[color=black] (-1.2744098530035017,2.585182533804536) node {$X_{T_0}$};
\draw[color=black] (-1.8,0.45) node {$v_0$};
\draw [fill=black] (1.7298151282607106,-0.7365751697717041) circle (1pt);
\draw[color=black] (2.2,-1.0137603618375217) node {$X_{T_1}$};
\draw[color=black] (-0.35,2.) node {$V_{T_0}$};
\draw[color=black] (0.8,-0.8) node {$V_{T_1}$};

\end{tikzpicture}
\caption{A trajectory of the stochastic billiard in a set $K$, starting in the interior of $K$}
\label{Figure billard}
\end{center}
\end{figure}

\begin{Rema}
We can obviously define the continuous-time process starting at any $x_0\in K$, what will in fact often do in this paper. If $x_0\in K\setminus \partial K$, we have to precise also the initial speed $v_0\in \SS^1$, and we can use the same scheme to construct the process.
\end{Rema}

For $x\in\partial K$, it is equivalent to consider the new speed in $\SS_x$ or to consider the angle in $\left[-\frac{\pi}{2},\frac{\pi}{2}\right]$ between this vector speed and the normal vector $n_x$. For $n\geq 1$, we thus denote by $\Theta_n$ the random variable in $\left[-\frac{\pi}{2},\frac{\pi}{2}\right]$ such that $r_{X_{T_n},\Theta_n}(n_{X_{T_n}})\overset{\mathcal{L}}{=}V_{T_n}$, where for $x\in\partial K$ and $\theta\in\RR$, $r_{x,\theta}$ denotes the rotation with center $x$ and angle $\theta$.\\
We make the following assumption on $\gamma$ (see Figure \ref{intervalle J}):\\

\textbf{Assumption $(\mathcal{H})$:}
\begin{quote} 
The law $\gamma$ has a density function $\rho$ with respect to the Haar measure on $\SS_e$, which satisfies: there exist $\mathcal{J}\subset \mathbb{S}_e$ symmetric with respect to $e$ and $\rho_{\min}>0$ such that:
\[\rho(u)\geq \rho_{\min}, ~ \text{ for all } u\in\mathcal{J}.\]
\end{quote}
This assumption is equivalent to the following one on the variables $(\Theta_n)_{n\geq 0}$:\\

\textbf{Assumption $(\mathcal{H}')$:}
\begin{quote} 
The variables $\Theta_n$, $n\geq 0$, have a density function $f$ with respect to the Lebesgue measure on $\left[ -\frac{\pi}{2},\frac{\pi}{2}\right]$ satisfying: there exist $f_{\min}>0$ and $\theta^*\in \left( 0, \frac{\pi}{2}\right)$ such that:
\begin{equation*}
f(\theta)\geq f_{\min}, ~ \text{ for all } \theta\in \left[ -\frac{\theta^*}{2},\frac{\theta^*}{2}\right].
\end{equation*}
\end{quote}

In fact, since these two assumptions are equivalent, we have
\begin{equation*}
\rho_{\min}=f_{\min} ~~ \text{ and } ~~ \vert \mathcal{J} \lvert = \theta^*.
\end{equation*}
In the sequel, we will use both descriptions of the speed vector  depending on which is the most suitable.

\begin{figure}
\begin{center}
\definecolor{qqwuqq}{rgb}{0.,0.39215686274509803,0.}
\definecolor{uuuuuu}{rgb}{0.26666666666666666,0.26666666666666666,0.26666666666666666}
\definecolor{ududff}{rgb}{0.30196078431372547,0.30196078431372547,1.}
\begin{tikzpicture}[line cap=round,line join=round,>=triangle 45,x=1.0cm,y=1.0cm]
\clip(-4.3,-3.08) rectangle (4.3,4);
\draw [shift={(0.,0.)},line width=1.pt] (0,0) -- (86.8035915471075:0.6) arc (86.8035915471075:116.80359154710752:0.6) -- cycle;
\draw [shift={(0.,0.)},line width=1.pt] (0,0) -- (56.8035915471075:0.6) arc (56.8035915471075:86.80359154710749:0.6) -- cycle;
\draw [shift={(0.29203893805309744,5.229380530973452)},line width=2.pt,color=blue]  plot[domain=4.25944071501971:5.147559936171003,variable=\t]({1.*5.2375287568722095*cos(\t r)+0.*5.2375287568722095*sin(\t r)},{0.*5.2375287568722095*cos(\t r)+1.*5.2375287568722095*sin(\t r)});
\draw [dashed,line width=1.pt,domain=-4.3:5] plot(\x,{(-0.--0.29203893805309744*\x)/-5.229380530973452});
\draw [dashed,line width=1.pt,domain=-4.3:0.2] plot(\x,{(-0.-5.229380530973452*\x)/-0.29203893805309744});
\draw [->,line width=1.pt] (0.,0.) -- (0.10175862101003531,1.8221356203939967);
\draw [shift={(0.,0.)},line width=3.pt]  plot[domain=0.9914096994550458:2.0386072506516433,variable=\t]({1.*1.824974804225659*cos(\t r)+0.*1.824974804225659*sin(\t r)},{0.*1.824974804225659*cos(\t r)+1.*1.824974804225659*sin(\t r)});
\draw [dotted,line width=1.pt,domain=-2:0.0] plot(\x,{(-0.--1.6288950469067374*\x)/-0.8229422593482347});
\draw [dotted,line width=1.pt,domain=0.0:2.5] plot(\x,{(-0.--1.5271364258967022*\x)/0.9991933610457617});
\draw [dotted,line width=1.pt] (0.,0.) circle (1.824974804225658cm);

\draw [fill] (0.,0.) circle (2.0pt);
\draw (0.2,-0.3) node {$x$};
\draw (0.3,0.8) node {$\theta^*$};
\draw[color=black] (0.6,2.13) node {$U_x\mathcal{J}$};
\draw[color=blue] (2.7,0.8) node {$\partial K$};
\draw (-0.2,1.2) node {$n_x$};

\end{tikzpicture}
\end{center}
\caption{Illustration of Assumptions $(\mathcal{H})$ and $(\mathcal{A})$ }
\label{intervalle J}
\end{figure}
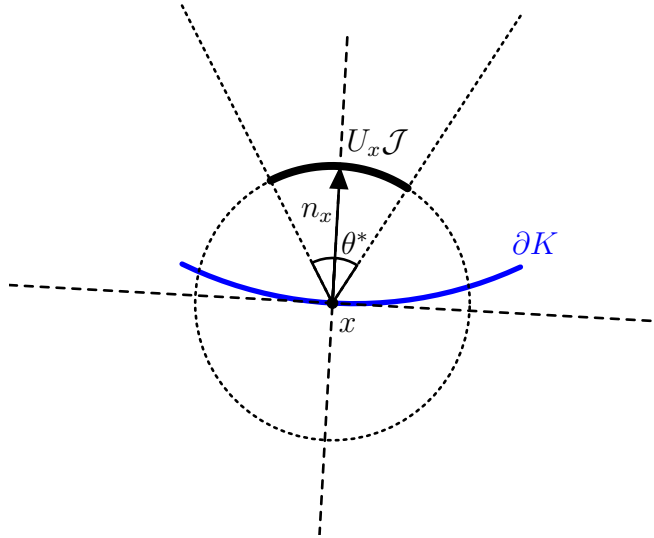

\subsection{A coupling for the stochastic billiard}\label{Subsection A coupling for the stochastic billiard}
Let us now informally describe the idea of the couplings used to explicit the speeds of convergence of our processes to equilibrium. They will be explain explicitly in Sections \ref{Section Stochastic billiard in the disc} and \ref{Section Stochastic billiard in a convex set with bounded curvature}.\\ \\
To get a bound on the speed of convergence of the Markov chain recording the location of hitting points on the boundary of the stochastic billiard, the strategy is the following. We consider two stochastic billiard Markov chains with different initial conditions. We estimate the number of steps that they have to do before they have a strictly positive probability to arrive on the same place at a same step. In particular, it is sufficient to know the number of steps needed before the position of each chain charges the half of the boundary of the set on which they evolve. Then, their coupling time is stochastically smaller than a geometric time whose Laplace transform is known.\\
The case of the continuous-time process is a bit more complicated. To couple two stochastic billiards, it is not sufficient to make them cross in the interior of the set where they live. Indeed, if they cross with a different speed, then they will not be equal after. So the strategy is to make them arrive at the same place on the boundary of the set at the same time, and then they can always keep the same velocity and stay equal. We will do this in two steps. First, we will make the two processes hit the boundary at the same time, but not necessary at the same point. This will take some random time, that we will be able to quantify. And secondly, with some strictly positive probability, after two bounces, the two processes will have hit the boundary at the same point at the same time. We repeat the scheme until the first success. This leads us to a stochastic upper bound for the coupling time of two stochastic billiards.\medskip\\
Obviously, the way that we couple our processes is only one way to do that, and there are many as we want. Let us for instance describe the coupling constructed in \cite{CPSV}. Consider two stochastic billiard processes evolving in the set $K$ with different initial conditions. Their first step is to make the processes hit the boundary in the neighbour of a good $x_1\in \partial K$. This can be done after $n_0$ bounces, where $n_0$ is the minimum number of bounces needed to connect any two points of the boundary of $K$. Once the two processes have succeed, they are in the neighbour of $x_1$, but at different times. Then, the strategy used by the authors of \cite{CPSV} is to make the two processes do goings and comings between the neighbour of $x_1$ and the neighbour of another good $y_1\in \partial K$. Thereby, if the point $y_1$ is well chosen, the time difference between the two processes decreases gradually, while the positions of the processes stay the same after one going and coming. However, the number of goings and comings needed to compensate for the possibly big difference of times could be very high. This particular coupling is therefore well adapted for sets whose boundary can be quite "chaotic", but not for convex sets with smooth boundary as we consider in this paper.

\section{Stochastic billiard in the disc}\label{Section Stochastic billiard in the disc}

In this section, we consider the particular case where $K$ is a ball: $K=\Br$, for some fixed $r>0$.\\
In that case, for each $n\geq 0$, the couple $(X_{T_n},V_{T_n})\in \partial \Br\times\S$ can be represented by a couple $(\Phi_n,\Theta_n)\in [0,2\pi)\times \left[-\frac{\pi}{2},\frac{\pi}{2}\right]$ as follows (see Figure \ref{dessincorrespondancecercle}):
\begin{itemize}
\item to a position $x$ on $\partial\Br$ corresponds an unique angle $\phi\in [0,2\pi)$. The variable $\Phi_n$ nominates this unique angle associated to $X_n$, i.e. $(1,\Phi_n)$ are the polar coordinate of $X_n$.
\item at each speed $V_{T_n}$ we associate the variable $\Theta_n$ introduced in Section \ref{Subsection Description of the process}, satisfying Assumption $(\mathcal{H}')$.
\end{itemize}
Remark that for all $n\geq 0$, the random variable $\Theta_n$ is independent of $\Phi_k$ for all $k\in\{0,n\}$. We also recall that the variables $\Theta_n$, $n\geq 0$, are all independent.

\begin{figure}
\begin{center}
\begin{tikzpicture}[line cap=round,line join=round,>=triangle 45,x=1.0cm,y=1.0cm,scale=0.8]
\draw[->,color=black] (-4,0.) -- (4,0.);
\foreach \x in {-6.,-5.,-4.,-3.,-2.,-1.,1.,2.,3.,4.,5.,6.,7.,8.}
\draw[shift={(\x,0)},color=black] (0pt,-2pt);
\draw[->,color=black] (0.,-4) -- (0.,4);
\foreach \y in {-5.,-4.,-3.,-2.,-1.,1.,2.,3.,4.}
\draw[shift={(0,\y)},color=black] (-2pt,0pt);
\clip(-5,-5) rectangle (5,5);
\draw [line width=1.5pt] (0.:0.6) arc (0.:141.6942291318932:0.6) ;
\draw [shift={(-2.3541418260183855,1.85957421551032)},line width=1.5pt]  (-78.58430135864899:0.6) arc (-78.58430135864899:-38.30577086810682:0.6);
\draw [line width=2.pt] (0.,0.) circle (3.cm);
\draw [->,line width=1pt] (-2.3541418260183855,1.85957421551032) -- (-2.12,0.7);
\draw [dashed,line width=1.pt] (-2.3541418260183855,1.85957421551032)-- (0.,0.);
\draw [fill] (-2.3541418260183855,1.85957421551032) circle (2.5pt);
\draw (-2.8,2.) node {$X_{T_n}$};
\draw (-2.45,0.7) node {$V_{T_n}$};
\draw  (0.6,0.7) node {$\Phi_n$};
\draw (-1.64,1) node {$\Theta_n$};
\end{tikzpicture}
\end{center}
\caption{Definition of the variables $\Phi_n$ and $\Theta_n$ in bijection with the variables $X_{T_n}$ and $V_{T_n}$}
\label{dessincorrespondancecercle}
\end{figure}
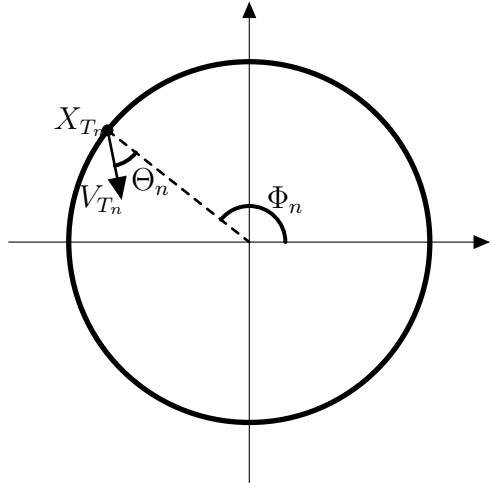
\medskip
In the sequel, we do not care about the congruence modulo $2\pi$ : it is implicit that when we write $\Phi$, we consider its representative in $[0,2\pi)$. \\
 
Let us state the following proposition that links the different random variables together.

\begin{Prop}\label{Prop lien entre les variables dans le cercle}
For all $n\geq 1$ we have:
\begin{equation}\label{Equation qui lie Phi_n, Theta_n, tau_n}
\tau_n=2r\cos(\Theta_{n-1}) ~~~~ \text{and} ~~~~ \Phi_n=\pi+2\Theta_{n-1}+\Phi_{n-1}
\end{equation}
\end{Prop}
\begin{proof}
The relationships are immediate with geometric considerations.
\end{proof}

\subsection{The embedded Markov chain}

In this section, the goal is to obtain a control of the speed of convergence of the stochastic billiard Markov chain on the circle. For this purpose, we study the distribution of the position of the Markov chain at each step.\\

Let $\Phi_0=\phi_0\in [0,2\pi)$.

\begin{Prop}\label{Prop densite des phi_n}
Let $(\Phi_n)_{n\geq 0}$ be the stochastic billiard Markov chain evolving on $\partial B_r$, satisfying assumption $(\mathcal{H}')$.\\
We have
\begin{equation*}
f_{\Phi_1}(u)\geq \frac{f_{\min}}{2},~~~ \forall u\in \mathcal{I}_1=\left[\pi-\theta^*+\phi_0,\pi+\theta^*+\phi_0\right].
\end{equation*}
Moreover, for all $n\geq 2$, for all $\eta_2,\cdots,\eta_n$ such that $\eta_2\in\left( 0 ,2\theta^*\right)$, and for $k\in\{2,\cdots,n-1\}$, $\eta_{k+1}\in \left(0,n\theta^* - \sum_{k=2}^{n-1} \eta_k \right)$, we have
\begin{multline*}
f_{\Phi_n}(u)\geq \left(\frac{f_{\min}}{2}\right)^{n}\eta_n\cdots \eta_2,\\ ~~~ \forall u\in \mathcal{I}_n=\left[(n-1)\pi-n\theta^*+\phi_0 + \sum_{k=2}^{n}\eta_k,(n-1)\pi+n\theta^*+\phi_0-\sum_{k=2}^{n}\eta_k \right].
\end{multline*}
\end{Prop}

\begin{proof}
Since the Markov chain is totally symmetric, we do the computations with $\phi_0=0$.
\begin{itemize}
\item Case $n=2$:\\
We have, thanks to \eqref{Equation qui lie Phi_n, Theta_n, tau_n}, $\Phi_1=\pi+2\Theta_0+\phi_0=\pi+2\Theta_0$. Thus, for any measurable bounded function $g$, we get:
\begin{align*}
\EE\left[g(\Phi_1) \right] &= \EE\left[g\left( \pi+2\Theta_0\right)\right] = \int_{-\frac{\pi}{2}}^{\frac{\pi}{2}} g\left( \pi+2x\right)f(x)\dt x\\
&\geq f_{\min}\int_{-\frac{\theta^*}{2}}^{\frac{\theta^*}{2}} g\left( \pi+2x\right)\dt x =\frac{f_{\min}}{2}\int_{\pi-\theta^*}^{\pi+\theta^*} g(u) \dt u.
\end{align*}
We deduce:
\begin{equation*}
f_{\Phi_1}(u)\geq \frac{f_{\min}}{2},~~~ \forall u\in \left[\pi-\theta^*,\pi+\theta^*\right].
\end{equation*}

\item Induction: let suppose that for some $n\geq 1$, $f_{\Phi_n}(u)\geq c_n$ for all $u\in \left[a_n,b_n \right]$. Then, using the relationship \eqref{Equation qui lie Phi_n, Theta_n, tau_n} and the independence between $\Theta_n$ and $\Phi_n$ we have, for any measurable bounded function $g$:
\begin{align*}
\EE\left[ g(\Phi_{n+1}) \right] &= \EE\left[ g(\pi+2\Theta_n+ \Phi_n)\right]\\
&\geq f_{\min}c_n\int_{-\frac{\theta^*}{2}}^{\frac{\theta^*}{2}} \int_{a_n}^{b_n} g(\pi+2\theta+x)\dt x \dt\theta.
\end{align*}
Using the substitution $u=\pi+2\theta+x$ in the integral with respect to $x$ and Fubini's theorem, we have:
\begin{align*}
\EE\left[ g(\Phi_{n+1}) \right] &\geq f_{\min}c_n \int_{\pi-\theta^*+a_n}^{\pi+\theta^*+b_n} \left(\int_{-\frac{\theta^*}{2}}^{\frac{\theta^*}{2}} \mathbf{1}_{\frac{1}{2}\left(u-\pi-b_n\right)\leq \theta\leq \frac{1}{2}\left(u-\pi-a_n\right)} \right) g(u)\dt u,
\end{align*}
and we deduce the following lower bound of the density function $f_{\Phi_{n+1}}$ of $\Phi_{n+1}$:
\begin{align*}
f_{\Phi_{n+1}}(u)\geq f_{\min}c_n\left\lvert \left[-\frac{\theta^*}{2},\frac{\theta^*}{2}\right] \cap \left[ \frac{1}{2}\left(u-\pi-b_n\right),\frac{1}{2}\left(u-\pi-a_n\right)\right]\right\lvert,
\end{align*}
for all $u\in \left[\pi-\theta^*+a_n,\pi+\theta^*+b_n\right]$.\\
When $u$ is equal to one extremal point of this interval, this lower bound is equal to $0$. However, let $\eta_{n+1}\in\left(0,\frac{1}{2}(b_n-a_n) \right)$, we have, for $u\in \left[\pi-\theta^*+a_n+\eta_{n+1},\pi+\theta^*+b_n-\eta_{n+1}\right]$:
\begin{align*}
f_{\Phi_{n+1}}(u)\geq f_{\min}c_n\frac{\eta_{n+1}}{2}.
\end{align*}

\end{itemize}
The results follows immediately.
\end{proof}

By choosing a constant sequence for the $\eta_k$, $k\geq 2$ in the Proposition \ref{Prop densite des phi_n}, we immediately deduce:

\begin{Coro}\label{Cor densite des Phi_n, intervalles J}
For all $n\geq 2$, for all $\varepsilon\in\left( 0,\theta^*\right)$, we have
\begin{equation*}
f_{\Phi_n}(u)\geq \left(\frac{f_{\min}}{2}\right)^{n}\varepsilon^{n-1},
\end{equation*}
$\forall u\in \mathcal{J}_n=\left[(n-1)\pi-n\theta^*+\phi_0 + (n-1)\varepsilon,(n-1)\pi+n\theta^*+\phi_0-(n-1)\varepsilon \right]$.
\end{Coro}
Let $(\mathcal{J}_n)_{n\geq 2}$ defined as in Corollary \ref{Cor densite des Phi_n, intervalles J}. We put $\mathcal{J}_1=\mathcal{I}_1$ with $\mathcal{I}_1$ defined in Proposition \ref{Prop densite des phi_n}.\\ 

\begin{Theo}
Let $(\Phi_n)_{n\geq 0}$ be the stochastic billiard Markov chain on the circle $\partial \Br$, satisfying assumption $(\mathcal{H}')$. \\
There exists a unique invariant probability measure $\nu$ on $[0,2\pi)$ for the Markov chain $(\Phi_n)_{n\geq 0}$, and we have:
\begin{enumerate}
\item if $\theta^*>\frac{\pi}{2}$, for all $n\geq 0$,
\begin{equation*}
\lVert \PP\left( \Phi_n\in \cdot\right) - \nu \lVert_{TV}\leq \left(1-f_{\min}(2\theta^*-\pi)\right)^{n-1},
\end{equation*}
\item  if $\theta^*\leq \frac{\pi}{2}$, for all $n\geq 0$ and all $\varepsilon\in \left(0,\theta^*\right)$,
\begin{equation*}
\lVert \PP\left( \Phi_n\in \cdot\right) - \nu \lVert_{TV}\leq \left(1-\alpha\right)^{\frac{n}{n_0}-1},
\end{equation*}
where
\begin{equation*}
n_0=\left\lfloor \frac{\pi-2\varepsilon}{2(\theta^*-\varepsilon)}\right\rfloor +1 ~~~ \text{and} ~~~ \alpha=\left(\frac{\varepsilon}{2}\right)^{n_0-1}{f_{\min}}^{n_0}\left( 2n_0\theta^*-2(n_0-1)\varepsilon -\pi\right).
\end{equation*}
\end{enumerate}

\end{Theo}

\begin{proof}
The existence of the invariant measure  is immediate thanks to the compactness of $\partial \Br$ (see \cite{EK}). The following proof leads to its uniqueness and the speed of convergence.\\
Let $(\Phi_n,\Theta_n)_{n\geq 0}$ and $(\tilde{\Phi}_n,\tilde{\Theta}_n)_{n\geq 0}$ be two versions of the process described above, with initial positions $\phi_0$ and $\tilde{\phi}_0$ on $\partial \Br$.\\
In order to couple $\Phi_n$ and $\tilde{\Phi}_n$ at some time $n$, it is sufficient to show that the intervals $\mathcal{J}_n$ and $\tilde{\mathcal{J}}_n$ corresponding to Corollary \ref{Cor densite des Phi_n, intervalles J} have a non empty intersection. Since these intervals are included in $[0,2\pi)$, a sufficient condition to have $\mathcal{J}_n \cap \tilde{\mathcal{J}}_n \neq \emptyset$ is that the length of these two intervals is strictly bigger than $\pi$.\medskip\\
Let $\varepsilon\in \left( 0, \theta^*\right)$. We have
\begin{equation*}
\lvert \mathcal{J}_1  \lvert= \lvert \tilde{\mathcal{J}}_1 \lvert = 2\theta^*,
\end{equation*}
and for $n\geq 2$,
\begin{equation*}
\lvert \mathcal{J}_n \lvert = \lvert \tilde{\mathcal{J}}_n \lvert = 2n\theta^*-2(n-1)\varepsilon.
\end{equation*}
Therefore the length of $\mathcal{J}_n$ is a strictly increasing function of $n$ (which in intuitively clear).
\begin{itemize}
\item Case 1: $\theta^*>\frac{\pi}{2}$. In that case we have $\lvert \mathcal{J}_1  \lvert= \lvert \tilde{\mathcal{J}}_1 \lvert >\pi$. Therefore we can construct a coupling $\left(\Phi_1,\tilde{\Phi}_1\right)$ such that we have, using Proposition \ref{Prop densite des phi_n}:
\begin{align*}
\PP\left( \Phi_1=\tilde{\Phi}_1\right) &\geq \frac{f_{\min}}{2}\left\lvert \mathcal{J}_1\cap \tilde{\mathcal{J}}_1 \right\lvert\\
&\geq\frac{f_{\min}}{2}2(2\theta^*-\pi)\\
&= f_{\min}(2\theta^*-\pi).
\end{align*}
\item Case 2: $\theta^*\leq \frac{\pi}{2}$. Here we need more jumps before having a positive probability to couple $\Phi_n$ and $\tilde{\Phi}_n$. Let thus define 
\begin{equation*}
n_0=\min\{n\geq 2: 2n\theta^*-2(n-1)\varepsilon >\pi\}=\left\lfloor \frac{\pi-2\varepsilon}{2(\theta^*-\varepsilon)}\right\rfloor +1.
\end{equation*}
Using the lower bound of the density function of $\Phi_{n_0}$ obtained in Corollary \ref{Cor densite des Phi_n, intervalles J}, we deduce that we can construct a coupling $\left( \Phi_{n_0},\tilde{\Phi}_{n_0}\right)$ such that:
\begin{align*}
\PP\left( \Phi_{n_0}=\tilde{\Phi}_{n_0} \right) &\geq \left( \frac{f_{\min}}{2}\right)^{n_0}\varepsilon^{n_0-1} \left\lvert \mathcal{J}_{n_0} \cap \tilde{\mathcal{J}}_{n_0} \right\lvert \\
&\geq \left( \frac{f_{\min}}{2}\right)^{n_0}\varepsilon^{n_0-1} 2\left( 2n_0\theta^*-2(n_0-1)\varepsilon -\pi\right)\\
&= \left(\frac{\varepsilon}{2}\right)^{n_0-1}\left(f_{\min}\right)^{n_0}\left(  2n_0\theta^*-2(n_0-1)\varepsilon-\pi\right).
\end{align*}
\end{itemize}
To treat both cases together, let define
\begin{equation*}
m_0= \mathbf{1}_{\theta^*>\frac{\pi}{2}}+ \left(\left\lfloor \frac{\pi-2\varepsilon}{2(\theta^*-\varepsilon)}\right\rfloor +1\right) \mathbf{1}_{\theta^*\leq \frac{\pi}{2}}.
\end{equation*}
and
\begin{equation*}
\alpha = f_{\min}(2\theta^*-\pi) \mathbf{1}_{\theta^*>\frac{\pi}{2}}+ \left(\frac{\varepsilon}{2}\right)^{m_0-1}(f_{\min})^{m_0}\left( 2m_0\theta^*-2(m_0-1)\varepsilon -\pi\right)\mathbf{1}_{\theta^*\leq \frac{\pi}{2}}.
\end{equation*}
We get:
\begin{align*}
\lVert \PP\left( \Phi_n\in \cdot\right) - \nu \lVert_{TV}&\leq \PP\left( \Phi_n \neq \tilde{\Phi}_n \right)\\
&\leq \PP\left( \Phi_{\lfloor \frac{n}{m_0}\rfloor m_0} \neq \tilde{\Phi}_{\lfloor \frac{n}{m_0}\rfloor m_0} \right)\\
&\leq \left(1-\alpha\right)^{\lfloor \frac{n}{m_0}\rfloor}\\
&\leq \left(1-\alpha\right)^{\frac{n}{m_0}-1}.
\end{align*}

\end{proof}

\subsection{The continuous-time process}

We assume here that the constant $\theta^*$ introduced in Assumption $(\mathcal{H}')$ satisfies
\begin{equation*}
\theta^*\in \left( \frac{2\pi}{3},\pi\right).
\end{equation*}
This condition on $\theta^*$ is essential in the proof of Theorem \ref{Thm convergence billiard avec theta*} to couple our processes with "two jumps". However, if $\theta^*\in\left(0,\frac{2\pi}{3}\right]$ we can adapt our method (see Remark \ref{Rem theta plus petit que 2pi/3}).\\

\textbf{Notation}: Let $x\in\partial \Br$. We write $T_n^x$ an $\Phi_n^x$ respectively for the hitting time of $\partial \Br$ and the position of the Markov chain after $n$ steps, and that started at position $x$.\\
Let us remark that the distribution of $T_n^x$ does not depend on $x$ since we consider here the stochastic billiard in the disc, which is rotationally symmetric. Therefore, we allow us to omit this $x$ when it is not necessary for the comprehension.

\begin{Prop}\label{Prop densite T2}
Let $(X_t,V_t)_{t\geq 0}$ be the stochastic billiard process in the ball $\Br$ satisfying Assumption $(\mathcal{H}')$ with $\theta^*\in\left( \frac{2\pi}{3},\pi\right)$.\\
We denote by $f_{T_2}$ the density function of $T_2$. Let $\eta\in \left(0,2r\left(1-\cos\left(\frac{\theta^*}{2}\right) \right)\right)$. We have
\[ f_{T_2}(x)\geq \delta ~~ \text{ for all } x\in [4r\cos\left(\frac{\theta^*}{2}\right)+\eta, 4r-\eta], \]
where
\begin{align}\label{minoration de fT_2 dans le cercle}
 \delta&= \frac{2f_{\min}^2}{r\sin\left(\frac{\theta^*}{2}\right)}\min \left\{   \frac{\theta^*}{2}-\arccos\left(
 \cos\left(\frac{\theta^*}{2}\right)+\frac{\eta}{2r} \right);\arccos\left(1-\frac{\eta}{2r} \right) \right\}.
\end{align}
\end{Prop}

\begin{proof}
If the density function $f$ is supported on $\left[-\frac{\theta^*}{2},\frac{\theta^*}{2}\right]$, it is immediate to observe that $4r\cos\left( \frac{\theta^*}{2}\right) \leq T_2 \leq 4r$. But let be more precise.\\
Let $g:\RR\rightarrow\RR$ be a bounded measurable function. Let us recall that, thanks to \eqref{Equation qui lie Phi_n, Theta_n, tau_n}, $T_2=2r\left( \cos(\Theta_1)+\cos(\Theta_2) \right)$ with $\Theta_1,\Theta_2$ two independent random variables with density function $f$. We have, using Assumption $\left( \mathcal{H}'\right)$:

\begin{align*}
\EE\left[ g(T_2) \right] &= \EE\left[ g\left(2r\left( \cos(\Theta_1)+\cos(\Theta_2) \right)\right) \right]\\
&\geq f_{\min}^2 \int_{-\frac{\theta^*}{2}}^{\frac{\theta^*}{2}} \int_{-\frac{\theta^*}{2}}^{\frac{\theta^*}{2}} g\left( 2r\left( \cos(u)+\cos(v) \right) \right)\dt u \dt v\\
&= 4f_{\min}^2 \int_{0}^{\frac{\theta^*}{2}} \int_{0}^{\frac{\theta^*}{2}} g\left( 2r\left( \cos(u)+\cos(v) \right) \right)\dt u \dt v.
\end{align*}
The substitution $x=2r\left(\cos(u)+\cos(v)\right)$ in the integral with respect to $u$ gives then:
\begin{align*}
\EE\left[ g(T_2) \right] &\geq 4f_{\min}^2 \int_{0}^{\frac{\theta^*}{2}} \int_{2r\left(\cos\left(\frac{\theta^*}{2}\right)+\cos(v)\right)}^{2r(1+\cos(v))} g(x) \frac{1}{2r\sin\left( \arccos\left(\frac{x}{2r}-\cos(v) \right) \right)}\dt x \dt v.
\end{align*}
Fubini's theorem leads to
\begin{align*}
\EE\left[ g(T_2) \right] &\geq \frac{2f_{\min}^2}{r} \int_{4r\cos\left( \frac{\theta^*}{2} \right)}^{4r} \left( \int_0^{\frac{\theta^*}{2}} \frac{1}{\sqrt{1-\left(\frac{x}{2r}-\cos(v) \right)^2}} \mathbf{1}_{\frac{x}{2r}-1<\cos(v)<\frac{x}{2r}-\cos\left( \frac{\theta^*}{2} \right)} \dt v\right) g(x) \dt x.
\end{align*}
We then deduce a lower-bound for the density function of $T_2$:
\begin{equation*}
f_{T_2}(x) \geq\frac{2f_{\min}^2}{r}  \int_0^{\frac{\theta^*}{2}} \frac{1}{\sqrt{1-\left(\frac{x}{2r}-\cos(v) \right)^2}} \mathbf{1}_{\frac{x}{2r}-1<\cos(v)<\frac{x}{2r}-\cos\left( \frac{\theta^*}{2} \right)} \dt v \mathbf{1}_{x\in \left(4r\cos\left( \frac{\theta^*}{2} \right),4r \right)}.
\end{equation*}
Let $x\in \left(4r\cos\left( \frac{\theta^*}{2} \right),4r \right)$. Cutting the interval $\left(4r\cos\left( \frac{\theta^*}{2} \right),4r \right)$ at point $2r\left(1+\cos\left(\frac{\theta^*}{2}\right)\right)$, we get: 
\begin{align*}
f_{T_2}(x) &\geq \frac{2f_{\min}^2}{r}  \int_0^{\frac{\theta^*}{2}} \frac{1}{\sqrt{1-\left(\frac{x}{2r}-\cos(v) \right)^2}} \mathbf{1}_{\frac{x}{2r}-1<\cos(v)<\frac{x}{2r}-\cos\left( \frac{\theta^*}{2} \right)} \dt v  \mathbf{1}_{x\in \left(4r\cos\left(\frac{\theta^*}{2}\right),2r\left(1+\cos\left(\frac{\theta^*}{2}\right)\right)\right]}\\
&\hspace*{1cm} +\frac{2f_{\min}^2}{r}  \int_0^{\frac{\theta^*}{2}} \frac{1}{\sqrt{1-\left(\frac{x}{2r}-\cos(v) \right)^2}} \mathbf{1}_{\frac{x}{2r}-1<\cos(v)<\frac{x}{2r}-\cos\left( \frac{\theta^*}{2} \right)} \dt v  \mathbf{1}_{x\in \left[2r\left( 1+\cos\left(\frac{\theta^*}{2}\right) \right),4r\right)}\\
&=\frac{2f_{\min}^2}{r}  \int_{\arccos\left(\frac{x}{2r}-\cos\left(\frac{\theta^*}{2}\right)\right)}^{\frac{\theta^*}{2}} \frac{1}{\sqrt{1-\left(\frac{x}{2r}-\cos(v) \right)^2}}  \dt v  \mathbf{1}_{x\in \left(4r\cos\left(\frac{\theta^*}{2}\right),2r\left(1+\cos\left(\frac{\theta^*}{2}\right)\right)\right]}\\
&\hspace*{1cm} + \frac{2f_{\min}^2}{r}  \int_{0}^{\arccos\left(\frac{x}{2r}-1 \right)} \frac{1}{\sqrt{1-\left(\frac{x}{2r}-\cos(v) \right)^2}} \dt v  \mathbf{1}_{x\in \left[2r\left( 1+\cos\left(\frac{\theta^*}{2}\right) \right),4r\right)}.
\end{align*}
Then, for $v\in \left(\arccos\left(\frac{x}{2r}-\cos\left(\frac{\theta^*}{2}\right)\right),\frac{\theta^*}{2}\right)$ we have $\cos(v)\leq \frac{x}{2r}-\cos\left(\frac{\theta^*}{2}\right)$, and for \\$v\in \left( 0,\arccos\left(\frac{x}{2r}-1 \right)\right)$ we have $\cos(v)\leq 1$. We thus have:
\begin{align*}
f_{T_2}(x) &\geq \frac{2f_{\min}^2}{r\sin\left(\frac{\theta^*}{2}\right)} \left( \frac{\theta^*}{2}-\arccos\left(\frac{x}{2r}-\cos\left(\frac{\theta^*}{2}\right) \right) \right) \mathbf{1}_{x\in \left(4r\cos\left(\frac{\theta^*}{2}\right),2r\left(1+\cos\left(\frac{\theta^*}{2}\right)\right)\right]}\\
&\hspace*{2cm} +\frac{2f_{\min}^2}{r} \frac{\arccos\left(\frac{x}{2r}-1 \right)}{\sqrt{\frac{x}{r}\left( 1-\frac{x}{4r}\right)}}  \mathbf{1}_{x\in \left[ 2r\left( 1+\cos\left(\frac{\theta^*}{2}\right)\right),4r \right)}.
\end{align*}
We can observe than the lower bound of $f_{T_2}$ is strictly positive for $x\in \left(4r\cos\left( \frac{\theta^*}{2} \right),4r \right)$, but is equal to $0$ when $x$ is one of the extremal points of this interval. Let therefore introduce $\eta\in \left(0,2r\left(1-\cos\left(\frac{\theta^*}{2}\right) \right)\right)$. We have:
\begin{itemize}
\item for $x\in [4r\cos\left(\frac{\theta^*}{2}\right)+\eta,2r\left( 1+\cos\left(\frac{\theta^*}{2}\right)\right)]$ we have
\begin{align*}
\frac{2f_{\min}^2}{r\sin\left(\frac{\theta^*}{2}\right)} &\left( \frac{\theta^*}{2}-\arccos\left(\frac{x}{2r}-\cos\left(\frac{\theta^*}{2}\right) \right) \right)\\
&\geq \frac{2f_{\min}^2}{r\sin\left(\frac{\theta^*}{2}\right)} \left( \frac{\theta^*}{2}-\arccos\left(\frac{4r\cos\left(\frac{\theta^*}{2}\right)+\eta}{2r}-\cos\left(\frac{\theta^*}{2}\right) \right) \right)\\
&= \frac{2f_{\min}^2}{r\sin\left(\frac{\theta^*}{2}\right)} \left( \frac{\theta^*}{2}-\arccos\left(\cos\left(\frac{\theta^*}{2}\right) +\frac{\eta}{2r}\right) \right)
\end{align*}
 \item for $x\in \left[ 2r\left( 1+\cos\left(\frac{\theta^*}{2}\right)\right),4r-\eta \right]$ we have
 \begin{align*}
 \frac{2f_{\min}^2}{r} \frac{\arccos\left(\frac{x}{2r}-1 \right)}{\sqrt{\frac{x}{r}\left( 1-\frac{x}{4r}\right)}}  &\geq \frac{2f_{\min}^2}{r} \frac{\arccos\left(\frac{4r-\eta}{2r}-1 \right)}{\sqrt{\frac{2r\left( 1+\cos\left(\frac{\theta^*}{2}\right)\right)}{r}\left( 1-\frac{2r\left( 1+\cos\left(\frac{\theta^*}{2}\right)\right)}{4r}\right)}} \\
 &= \frac{2f_{\min}^2}{r} \frac{\arccos\left(1-\frac{\eta}{2r} \right)}{\sqrt{\left( 1+\cos\left(\frac{\theta^*}{2}\right)\right)\left( 1-\cos\left(\frac{\theta^*}{2}\right)\right)}}\\
 &= \frac{2f_{\min}^2}{r\sin\left(\frac{\theta^*}{2}\right)}\arccos\left(1-\frac{\eta}{2r} \right).
 \end{align*}
\end{itemize}
The result follows immediately.

\end{proof}

\textbf{Notation}: For $x\in\partial\Br$, we denote by $\varphi_x$ the unique angle in $[0,2\pi)$ describing the position of $x$ on $\partial\Br$.

\begin{Prop}\label{Prop continuite (angle,temps)}
Let $(X_t,V_t)_{t\geq 0}$ be the stochastic billiard process in $\Br$ satisfying Assumption $(\mathcal{H}')$ with $\theta^*\in\left( \frac{2\pi}{3},\pi\right)$.\\
For all $\varepsilon\in \left( 0,\frac{\theta^*}{4}\right)$, the pair $\left( \Phi_2^x,T_2^x\right)$ is $\frac{f_{\min}^2}{2r\sin\left(\frac{\theta^*}{4}\right)}$-continuous on $(\varphi_x-\theta^*+4\varepsilon,\varphi_x+\theta^*-4\varepsilon)\times \left(2r\cos\left(\frac{\theta^*}{4}\right),2r\cos\left(\frac{\theta^*}{4}-\varepsilon\right)\right)$ for all $x\in\partial\mathcal{B}(0,r)$.
\end{Prop}

\begin{proof}
By symmetry of the process, it is sufficient to prove the lemma for $x\in \partial\Br$ such that $\varphi_x=0$, what we do.\\
Let $\varepsilon\in \left( 0,\frac{\pi}{4}\right)$, $A\subset (-\theta^*+4\varepsilon,\theta^*-4\varepsilon)$ and $(r_1,r_2)\subset \left(2r\cos\left(\frac{\theta^*}{4}\right),2r\cos\left(\frac{\theta^*}{4}-\varepsilon\right)\right)$.\\
Let us recall that $ \Phi_{2}^0=2\Theta_1+2\Theta_2$ and $T_2^0=2r(\cos(\Theta_1)+\cos(\Theta_2))$, where $\Theta_1,\Theta_2$ are independent variables with density function $f$. We thus have:
\begin{align*}
\PP &\left( \Phi_{2}^0\in A, T_2^0\in (r_1,r_2) \right)\\
&= \PP \left(2\Theta_1+2\Theta_2\in A, 2r(\cos(\Theta_1)+\cos(\Theta_2))\in (r_1,r_2) \right)\\
&= \int_{-\frac{\pi}{2}}^{\frac{\pi}{2}}\int_{-\frac{\pi}{2}}^{\frac{\pi}{2}} \mathbf{1}_{2u+2v\in A}\mathbf{1}_{\cos(u)+\cos(v)\in \left(\frac{r_1}{2r},\frac{r_2}{2r}\right)}f(u)f(v)\dt u \dt v\\
&\geq f_{\min}^2\int_{-\frac{\theta^*}{2}}^{\frac{\theta^*}{2}}\int_{-\frac{\theta^*}{2}}^{\frac{\theta^*}{2}} \mathbf{1}_{\frac{u+v}{2}\in \frac{A}{4}}\mathbf{1}_{\cos\left(\frac{u+v}{2}\right)\cos\left(\frac{u-v}{2}\right)\in \left(\frac{r_1}{4r},\frac{r_2}{4r}\right)}\dt u \dt v.
\end{align*}
Let us consider 
\[ g~:~(u,v)\in \left[-\frac{\theta^*}{2},\frac{\theta^*}{2}\right]^2 \longmapsto \left( \frac{u+v}{2}, \frac{u-v}{2} \right). \]
We have 
\[ \left[-\frac{\theta^*}{4},\frac{\theta^*}{4}\right]^2\subset g \left(\left[-\frac{\theta^*}{2},\frac{\theta^*}{2}\right]^2\right),\]
and 
\[ \left\lvert \det \mathrm{Jac}_g \right\lvert = \frac{1}{2}. \]
With this substitution, and using Fubini's theorem, we get:
\begin{align*}
\PP &\left( \Phi_{2}^0\in A, T_2^0\in (r_1,r_2) \right)\\
&\geq f_{\min}^2 \int_{-\frac{\theta^*}{4}}^{\frac{\theta^*}{4}}\int_{-\frac{\theta^*}{4}}^{\frac{\theta^*}{4}} \mathbf{1}_{x\in \frac{A}{4}}\mathbf{1}_{\cos(x)\cos(y)\in \left( \frac{r_1}{4r},\frac{r_2}{4r} \right)}2\dt x \dt y\\
&= 4f_{\min}^2 \int_{-\frac{\theta^*}{4}}^{\frac{\theta^*}{4}}\int_0^{\frac{\theta^*}{4}} \mathbf{1}_{\cos(x)\cos(y)\in \left( \frac{r_1}{4r},\frac{r_2}{4r} \right)}\dt y \mathbf{1}_{x\in \frac{A}{4}} \dt x.
\end{align*}
We now do the substitution $z=\cos(x)\cos(y)$ in the integral with respect to $\dt y$:
\begin{align*}
\PP &\left( \Phi_{2}^0\in A, T_2^0\in (r_1,r_2) \right)\\
&\geq 4f_{\min}^2 \int_{-\frac{\theta^*}{4}}^{\frac{\theta^*}{4}} \int_{\cos\left(\frac{\theta^*}{4}\right)\cos(x)}^{\cos(x)} \mathbf{1}_{z\in \left( \frac{r_1}{2r},\frac{r_2}{2r} \right)}\frac{1}{\sqrt{\cos^2(x)-z^2}} \dt z \mathbf{1}_{x\in \frac{A}{4}}\dt x\\
&\geq 4f_{\min}^2 \int_{-\frac{\theta^*}{4}}^{\frac{\theta^*}{4}} \int_{\cos\left(\frac{\theta^*}{4}\right)\cos(x)}^{\cos(x)} \mathbf{1}_{z\in \left( \frac{r_1}{2r},\frac{r_2}{2r} \right)}\frac{1}{\sin\left(\frac{\theta^*}{4}\right)} \dt z \mathbf{1}_{x\in \frac{A}{4}}\dt x\\
&\geq \frac{4f_{\min}^2}{\sin\left(\frac{\theta^*}{4}\right)}  \int_{-\frac{\theta^*}{4}+\varepsilon}^{\frac{\theta^*}{4}-\varepsilon} \int_{\cos\left(\frac{\theta^*}{4}\right)}^{\cos\left(\frac{\theta^*}{4}-\varepsilon\right)} \mathbf{1}_{z\in \left( \frac{r_1}{2r},\frac{r_2}{2r} \right)} \dt z \mathbf{1}_{x\in \frac{A}{4}}\dt x\\
&= \frac{f_{\min}^2}{2r\sin\left(\frac{\theta^*}{4}\right)}  (r_2-r_1)\left| A \right|,
\end{align*}
where we have used for the last equality the fact that $A\subset [-\theta^*+4\varepsilon,\theta^*-4\varepsilon)$ and $(r_1,r_2)\subset \left(2r\cos\left(\frac{\theta^*}{4}\right),2r\cos\left(\frac{\theta^*}{4}-\varepsilon\right)\right)$.\\
This ends the proof.
\end{proof}

Let fix $\eta\in (0,r\left(1-2\cos\left(\frac{\theta^*}{2}\right)\right))$ and $\varepsilon\in \left( 0,\frac{2\theta^*-\pi}{8} \right)$ (the condition $\theta^*>\frac{2\pi}{3}$ ensures that we can take such $\eta$ and $\varepsilon$).\\
Let define 
\begin{equation}\label{def h convergence du billard dans le cercle}
h= 4r\left(1-\cos\left(\frac{\theta^*}{2} \right)\right)-2\eta-2r= 2r\left( 1-2\cos\left( \frac{\theta^*}{2}\right)\right)-2\eta >0
\end{equation}
and 
\begin{align}\label{def alpha bis}
\alpha&=\frac{f_{\min}^2}{2r\sin\left(\frac{\theta^*}{4}\right)}(4\theta^*-2\pi-16\varepsilon)2r\left( \cos\left(\frac{\theta^*}{4}-\varepsilon\right)-\cos\left(\frac{\theta^*}{4}\right) \right) \nonumber \\
&= \frac{f_{\min}^2}{\sin\left(\frac{\theta^*}{4}\right)}(4\theta^*-2\pi-16\varepsilon)\left( \cos\left(\frac{\theta^*}{4}-\varepsilon\right)-\cos\left(\frac{\theta^*}{4}\right) \right)
\end{align}

\begin{Theo}\label{Thm convergence billiard avec theta*}
Let $(X_t,V_t)_{t\geq 0}$ be the stochastic billiard process in $B_r$ satisfying Assumption $(\mathcal{H}')$ with $\theta^*\in\left( \frac{2\pi}{3},\pi\right)$.\\
There exists a unique invariant probability measure on $\Br\times \S$ for the process $(X_t,V_t)_{t\geq 0}$.\\
Moreover let $\eta\in (0,r\left(1-2\cos\left(\frac{\theta^*}{2}\right)\right))$ and $\varepsilon\in \left( 0,\frac{2\theta^*-\pi}{8} \right)$. For all $t\geq 0$ and all $\lambda<\lambda_M$ we have
\begin{equation*}
\lVert \PP\left( X_t\in \cdot, V_t\in \cdot \right) - \chi \lVert_{TV} \leq C_{\lambda}\mathrm{e}^{-\lambda t},
\end{equation*}
where
\begin{equation}\label{Eq lambaM}
\lambda_M=\min\left\{\frac{1}{4r}\log\left( \frac{1}{1-\delta h}\right) ; \frac{1}{4r}\log\left( \frac{-(1-\delta h)+\sqrt{(1-\delta h)^2+4\delta h(1-\alpha)}}{2\delta h(1-\alpha)}\right) \right\}.
\end{equation}
and
\begin{equation*}
C_\lambda=\frac{\alpha\delta h \mathrm{e}^{10\lambda r}}{1-\mathrm{e}^{4\lambda r}(1-\delta h)-\mathrm{e}^{8\lambda r}\delta h(1-\alpha)},
\end{equation*}
with $\delta$, $h$ and $\alpha$ respectively given by \eqref{minoration de fT_2 dans le cercle}, \eqref{def h convergence du billard dans le cercle} and \eqref{def alpha bis}.
\end{Theo}

\begin{Rema}
The following proof of this theorem is largely inspired by the proof of Theorem 2.2 in \cite{CPSV}.
\end{Rema}

\begin{proof}
The existence of the invariant probability measure comes from the compactness of the space $\Br\times\SS^1$. The following proof show its uniqueness and gives the speed of convergence of the stochastic billiard to equilibrium.\\
Let $(X_t,V_t)_{t\geq 0}$ and $(\tilde{X}_t,\tilde{V}_t)_{t\geq 0}$ be two versions of the stochastic billiard with $(X_0,V_0)=(x_0,v_0)\in\Br\times\S$ and $(\tilde{X}_0,\tilde{V}_0)=(\tilde{x}_0,\tilde{v}_0)\in\Br\times\S$.\\
We recall the definition of $T_0$ and $\tilde{T}_0$ and define $w,\tilde{w}$ as follows:
\[ T_0=\inf\{ t\geq 0: x_0+tv_0 \notin K\}, ~~~~~ w=x_0+T_0 v_0\in\partial \Br,  \]
and
\[\tilde{T}_0=\inf\{ t\geq 0: \tilde{x}_0+t\tilde{v}_0 \notin K\}, ~~~~~ \tilde{w}=\tilde{x}_0+\tilde{T}_0 \tilde{v}_0\in\partial Br .\]
We are going to couple $(X_t,V_t)$ and $(\tilde{X}_t,\tilde{V}_t)$ in two steps: we first couple the times, so that the two processes hit $\partial\Br$ at a same time, and then we couple both position and time.\\
In the sequel, we write $X_{T_n}^a$ or $\tilde{X}_{T_n}^a$ for the position of the Markov chain at time $T_n$ when it starts at position $a\in\partial\Br$. Similarly, we write $T_n^a$ and $\tilde{T}_n^a$ for the successive hitting times of $\partial\Br$ of the processes.\medskip\\
\textbf{Step 1}. 
Proposition \ref{Prop densite T2} ensures that $T_2^w$ and $\tilde{T}_2^{\tilde{w}}$ are both $\delta$-continuous on $[4r\cos\left(\frac{\theta^*}{2}\right)+\eta,4r-\eta]$. Therefore, the variables $T_0+T_2^w$ and $\tilde{T}_0+\tilde{T}_2^{\tilde{w}}$ are $\delta$-continuous on \\$[T_0+4r\cos\left(\frac{\theta^*}{2}\right)+\eta,T_0+4r-\eta]\cap [\tilde{T}_0+4r\cos\left(\frac{\theta^*}{2}\right)+\eta,\tilde{T}_0+4r-\eta]$, with \\$\left\lvert [T_0+2r\cos\left(\frac{\theta^*}{2}\right)+\eta,T_0+4r-\eta]\cap [\tilde{T}_0+2r\cos\left(\frac{\theta^*}{2}\right)+\eta,\tilde{T}_0+4r-\eta] \right\lvert \geq h$ since $\lvert T_0-\tilde{T}_0\lvert \leq 2r$. Note that the condition $\theta^*>\frac{2\pi}{3}$ has been introduced to ensure that this intersection is non-empty.\\
Thus, there exists a coupling of $T_0+T_2^w$ and $\tilde{T}_0+\tilde{T}_2^{\tilde{w}}$ such that 
\[ \PP\left( E_1 \right) \geq \delta h, \]
where
\begin{equation*}
E_1=\left\{ T_0+ T_2^w = \tilde{T}_0+\tilde{T}_2^{\tilde{w}} \right\}.
\end{equation*}
On the event $E_1$ we define $T_c^1=T_0+T^w_2$.\\
On the event $E_1^c$, we can suppose, by symmetry that $T_0+ T_2^w \leq  \tilde{T}_0+\tilde{T}_2^{\tilde{w}}$. In order to try again to couple the hitting times, we need to begin at times whose difference is smaller than $2r$. Let thus  define
\begin{equation*}
m_1=\min\left\{ n>0: T_0+ T_2^w + T_n^{X_{T_0+ T_2^w}} > \tilde{T}_0+\tilde{T}_2^{\tilde{w}} \right\} ~~ \text{and} ~~ \tilde{m}_1=0.
\end{equation*}
We then have $\left\lvert \left(T_0+ T_2^w + T_{m_1}^{X_{T_0+ T_2^w}} \right)- \left(\tilde{T}_0 + \tilde{T}_{2}^{\tilde{w}}+\tilde{T}_{\tilde{m}_1}^{\tilde{X}_{\tilde{T}_0 + \tilde{T}_{2}^{\tilde{w}}}} \right) \right\lvert \leq 2r$.\\
Defining
\[ Z_0=X_{T_0+ T_2^w}, ~~~ Z_1=X_{T_0+ T_2^w + T_{m_1}^{Z_0}}, ~~~  \tilde{Z}_0=\tilde{X}_{\tilde{T}_0 + \tilde{T}_{2}^{\tilde{w}}}, ~~~ \tilde{Z}_1=\tilde{X}_{\tilde{T}_0 + \tilde{T}_{2}^{\tilde{w}}+\tilde{T}_{\tilde{m}_1}^{\tilde{Z}_0}} ,  \]
we obtain as previously:
\begin{equation*}
\PP\left( E_2 \lvert E_1^c \right) \geq \delta h,
\end{equation*}
where 
\begin{equation*}
E_2=\left\{ T_0+ T_2^w + T_{m_1}^{Z_0} + T_2^{Z_1} = \tilde{T}_0 + \tilde{T}_{2}^{\tilde{w}}+\tilde{T}_{\tilde{m}_1}^{\tilde{Z}_0} + \tilde{T}_2^{\tilde{Z}_1}  \right\}.
\end{equation*}
On the event $E_1^c\cap E_2$ we define $T^1_c=T_0+ T_2^w + T_{m_1}^{Z_0} + T_2^{Z_1}$. We thus have $T^1_c \overset{\mathcal{L}}{=} T_0 + R^1+R^2$, with $R^1,R^2$ independent variables with distribution $f_{T_2}$.\\
We then repeat the same procedure. We thus construct two sequences of stopping times $(m_k)_{k\geq 1}$, $(\tilde{m}_k)_{k\geq 1}$ and a sequence of events $(E_k)_{k\geq 1}$
satisfying
\begin{equation*}
\PP\left( E_k \lvert E_1^c\cap \cdots \cap E_{k-1}^c \right)\geq \delta h.
\end{equation*}
On the event $E_1^c\cap \cdots \cap E_{k-1}^c\cap E_k $ we define $T_c^1$ as previously, and we have $T_c^1\overset{\mathcal{L}}{=} T_0 + R^1+\cdots+R^k$ with $R^1,\cdots,R^k$ independent variables with distribution $f_{T_2}$. By construction, $T_c^1$ is the coupling time of the hitting times of the boundary.
\medskip \\
\textbf{Step 2}. Let us now work conditionally on $T^1_c$.\\
Let define $y=X_{T_c^1}^w$ and $\tilde{y}=\tilde{X}_{T_c^1}^{\tilde{w}}$ in order to simplify the notations. By construction of $T^1_c$, $y$ and $\tilde{y}$ are on $\partial\Br$. We define $N^1_c=\min\left\{n>0: X_{T_n}^w=y \right\}$, i.e. $T_c^1$ is the time at which the chain starting at $w$ hit the boundary for the $N_c^1$-th time.\\
Proposition \ref{Prop continuite (angle,temps)} ensures that the couples $\left( X_{T_2}^{y},T_2^{y} \right)$ and $\left( \tilde{X}_{T_2}^{\tilde{y}},\tilde{T}_2^{\tilde{Xy}} \right)$ are both $\frac{f_{\min}^2}{2r\sin\left(\frac{\theta^*}{4}\right)}$-continuous on the set $\left( (\varphi_{y}-\theta^*+4\varepsilon,\varphi_{y}+\theta^*-4\varepsilon)\cap (\varphi_{\tilde{y}}-\theta^*+4\varepsilon,\varphi_{\tilde{y}}+\theta^*-4\varepsilon) \right) \times \left(2\sqrt{2}r, 4\cos\left(\frac{\pi}{4}-\varepsilon\right)r \right)$, with $\left\lvert (\varphi_{y}-\theta^*+4\varepsilon,\varphi_{y}+\theta^*-4\varepsilon)\cap (\varphi_{\tilde{y}}-\theta^*+4\varepsilon,\varphi_{\tilde{y}}+\theta^*-4\varepsilon) \right\lvert \geq 4\theta^*-2\pi-16\varepsilon$. Note that the condition $\theta^*>\frac{2\pi}{3}$ implies in particular that the previous intersection in non-empty.\\
Therefore we can construct a coupling such that
\begin{equation*}
\PP\left( F \lvert E_1^c\cap\cdots\cap E_{N^1_c -1}^c\cap E_{N^1_c } \right) \geq \alpha,
\end{equation*}
where
\begin{equation*}
F= \left\{  X_{T_2}^{y}= \tilde{X}_{T_2}^{\tilde{y}} ~~ \text{and} ~~ T_2^{y}= \tilde{T}_2^{\tilde{y}} \right\}.
\end{equation*}
On the event $F$ we define $T_c=T^1_c+T_2^{y}$.\\
If $F$ does not occur, we can not directly try to couple both position and time since the two processes have not necessarily hit $\partial \Br$ at the same time. We thus have to couple first the hitting times, as we have done in step 1.\medskip\\
Let suppose that on $\left(E_1^c\cap\cdots\cap E_{N^1_c -1}^c\cap E_{N^1_c }\right)\cap F^c$, we have $T_2^{y}\leq \tilde{T}_2^{\tilde{y}}$ (the other case can be treated in the same way thanks to the symmetry of the problem). Let define
\begin{equation*}
\ell=\min\left\{n>0: T_2^{y} + T_n^{X_{T_2^{y}}} >  \tilde{T}_2^{\tilde{y}} \right\} ~~ \text{and} ~~ \tilde{\ell}=0
\end{equation*}
We clearly have $\left\lvert T_2^{y} + T_\ell^{X_{T_2^{y}}} - \left(\tilde{T}_{2}^{\tilde{y}} + \tilde{T}_{\tilde{\ell}}^{\tilde{X}_{\tilde{T}_{2}^{\tilde{y}}}}\right)\right\lvert\leq 2r$. Therefore, we can start again: we try to couple the times at which the two processes hit the boundary, and then to couple the positions and times together. \\ \\
Finally, the probability that we succeed to couple the positions and times in "one step" is:
\begin{align*}
\PP&\left( \left(\underset{k\geq 1}{\cup} \left( E_1^c\cap \cdots \cap E_{k-1}^c\cap E_k\right) \right) \cap F \right)\\
&= \PP\left( F \left\lvert \underset{k\geq 1}{\cup} \left( E_1^c\cap \cdots \cap E_{k-1}^c\cap E_k\right) \right. \right) \PP\left( \underset{k\geq 1}{\cup} \left( E_1^c\cap \cdots \cap E_{k-1}^c\cap E_k\right) \right)\\
&= \PP\left( F \left\lvert \underset{k\geq 1}{\cup} \left( E_1^c\cap \cdots \cap E_{k-1}^c\cap E_k\right) \right. \right)\\
&\geq\alpha.
\end{align*}
Thus, the coupling time $\hat{T}$ of the couples position-time satisfies:
\begin{align*}
\hat{T} &\leq_{st}  T_0+ \sum_{k=1}^{G} \left( \left(\sum_{l=1}^{G^k} T^{k,l} \right) + T^k \right)
\end{align*}
where $G\sim \mathcal{G}\left(\alpha \right)$, $G^1,G^2,\cdots\sim\mathcal{G}\left( \delta h \right)$ are independent geometric variables, and $\left(T^{k,l}\right)_{k,l\geq 1}$, $\left( T^k\right)_{k\geq 1}$ are independent random variables, independent from the geometric variables, with distribution $f_{T_2}$.\\
Let $\lambda \in \left( 0, \lambda_M\right)$, with $\lambda_M$ defined in equation \eqref{Eq lambaM}. Since all the random variables $T^{k,l}$ and $T^k$, $k,l\geq 1$, are almost surely smaller than two times the diameter of the ball $\Br$, and since $T_0$ is almost surely smaller than this diameter, we have:

\begin{align*}
\PP\left( \hat{T}>t \right) &\leq \mathrm{e}^{-\lambda t}\EE\left[ \mathrm{e}^{\lambda \hat{T}}\right]\\
&\leq  \mathrm{e}^{\lambda(T_0-t)}\EE\left[\exp\left(\lambda \sum_{k=1}^{G} \left( \left(\sum_{l=1}^{G^k} T^{k,l} \right) + T^k \right) \right)\right]\\
&\leq \mathrm{e}^{\lambda(2r-t)} \EE\left[ \prod_{k=1}^{G}\left( \left(\prod_{l=1}^{G^k} \exp\left( \lambda 4r\right)\right) \exp\left(\lambda 4r\right) \right) \right]\\
&= \mathrm{e}^{\lambda(2r-t)} \EE\left[ \prod_{k=1}^{G} \EE\left[ \mathrm{e}^{4\lambda r(G^k+1)} \right] \right].
\end{align*}
Now, using the expression of generating function of a geometric random variable we get:
\begin{align*}
\PP\left( \hat{T}>t \right) &\leq \mathrm{e}^{\lambda(2r-t)} \EE\left[ \prod_{k=1}^{G}\left(\sum_{l=1}^{\infty}\mathrm{e}^{4\lambda r(l+1)} \delta h (1-\delta h)^{l-1}\right) \right]\\
&= \mathrm{e}^{\lambda(2r-t)} \EE\left[ \left(\frac{\mathrm{e}^{8\lambda r}\delta h}{1-\mathrm{e}^{4\lambda r}(1-\delta h)}\right)^G \right]\\
&=  \mathrm{e}^{\lambda(2r-t)} \frac{\alpha\mathrm{e}^{8\lambda r}\delta h}{1-\mathrm{e}^{4\lambda r}(1-\delta h)}\frac{1}{1-\frac{\mathrm{e}^{8\lambda r}\delta h(1-\alpha)}{1-\mathrm{e}^{4\lambda r}(1-\delta h)}}\\
&= \mathrm{e}^{-\lambda t}\frac{\alpha\mathrm{e}^{10\lambda r}\delta h}{1-\mathrm{e}^{4\lambda r}(1-\delta h)-\mathrm{e}^{8\lambda r}\delta h(1-\alpha)}.
\end{align*}

This calculations are valid for $\lambda>0$ such that the generating functions are well defined, that is for $\lambda>0$ satisfying
\begin{equation*}
\mathrm{e}^{4\lambda r}(1-\delta h)<1 ~~ \text{and} ~~ \frac{\mathrm{e}^{8\lambda r}\delta h(1-\alpha)}{1-\mathrm{e}^{4\lambda r}(1-\delta h)}<1.
\end{equation*}
The first condition is equivalent to $\lambda<\frac{1}{4r}\log\left( \frac{1}{1-\delta h} \right)$.\\
The second condition is equivalent to $\delta h(1-\alpha)s^2+(1-\delta h)s -1 <0$ with $s=\mathrm{e}^{4\lambda r}$. It gives $s_1<s<s_2$ with $s_1=\frac{-(1-\delta h)-\sqrt{\Delta}}{2\delta h(1-\alpha)}<0$ and $s_2=\frac{-(1-\delta h)+\sqrt{\Delta}}{2\delta h(1-\alpha)}>1$ where $\Delta=(1-\delta h)^2+4\delta h(1-\alpha)>0$. And finally we get $\lambda<\frac{1}{4r}\log\left(s_2 \right)$.\\
Therefore, the estimation for $\PP\left( \hat{T}>t\right)$ is indeed valid for all $\lambda\in\left( 0, \lambda_M\right)$.
The conclusion of the theorem follows immediately.

\end{proof}

\begin{Rema}\label{Rem theta plus petit que 2pi/3}
If $\theta^*\in\left( 0,\frac{2\pi}{3}\right]$, Step $1$ of the proof of Theorem \ref{Thm convergence billiard avec theta*} fails: the intervals on which the random variables $T_0+T_2^w$ and $\tilde{T}_0+\tilde{T}_2^{\tilde{w}}$ are continuous can have an empty intersection. Similarly, in Step $2$, the intersection of the intervals on which the couples $\left( X_{T_2}^{X_{T_c^1}^w},T_2^{X_{T_c^1}^w} \right)$ and $\left( \tilde{X}_{T_2}^{\tilde{X}_{T_c^1}^{\tilde{w}}},\tilde{T}_2^{\tilde{X}_{T_c^1}^{\tilde{w}}} \right)$ are continuous can be empty if $\theta^*\leq \frac{\pi}{2}$.\\
However, instead of trying to couple the times or both positions and times in two jumps, we just need more jumps to do that. Therefore, the method and the results are similar in the case $\theta^*\leq \frac{2\pi}{3}$, the only difference is that the computations and notations will be much more awful.
\end{Rema}

\section{Stochastic billiard in a convex set with bounded curvature}\label{Section Stochastic billiard in a convex set with bounded curvature}

We make the following assumption on the set $K$ in which the stochastic billiard evolves:\\

\textbf{Assumption $(\mathcal{K})$:}
\begin{quote} 
$K$ is a compact convex set with curvature bounded from above by $C<\infty$ and bounded from below by $c>0$.
\end{quote}
This means that for each $x\in\partial K$, there is a ball $B_1$ with radius $\frac{1}{C}$ included in $K$ and a ball $B_2$ containing $K$, so that the tangent planes of $K$, $B_1$ and $B_2$ at $x$ coincide (see Figure \ref{Figure hypothese (K)}). In fact, for $x\in\partial K$, the ball $B_1$ is the ball with radius $\frac{1}{C}$ and with center the unique point at distance $\frac{1}{C}$ from $x$ in the direction of $n_x$. And $B_2$ is the one with the center at distance $\frac{1}{c}$ from $x$ in the direction of $n_x$.\\

\definecolor{ttqqqq}{rgb}{0.2,0.,0.}
\definecolor{qqqqff}{rgb}{0.,0.,1.}
\definecolor{ffqqqq}{rgb}{1.,0.,0.}
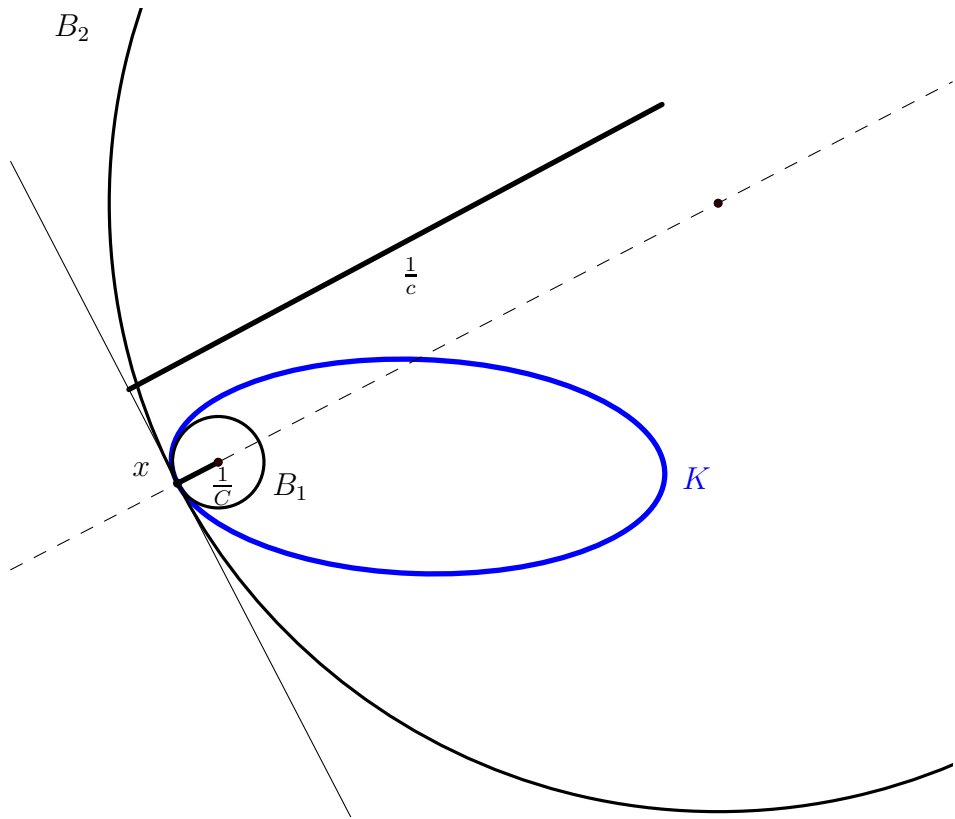
\begin{figure}
\begin{center}
\begin{tikzpicture}[line cap=round,line join=round,>=triangle 45,x=1.0cm,y=1.0cm]
\clip(-4.52,-3) rectangle (8.,7.76);
\draw [rotate around={177.89895737959574:(0.8971937941376569,1.6628121909862212)},line width=2.pt,color=qqqqff] (0.8971937941376569,1.6628121909862212) ellipse (3.281637848234518cm and 1.4246016723740447cm);
\draw [line width=0.4pt,domain=-4.52:8.] plot(\x,{(-645.9794430571505-415.84600120290315*\x)/215.60163868717063});
\draw [line width=0.4pt,dash pattern=on 5pt off 5pt,domain=-4.52:8.] plot(\x,{(--1094.702010712673--215.60163868717063*\x)/415.84600120290315});
\draw [line width=1.2pt] (-1.7599882128680713,1.7199772653319878) circle (0.6082762524907557cm);
\draw [line width=1.2pt] (4.886536348815252,5.165968288278606) circle (8.095007373566787cm);
\draw [line width=2.pt] (-2.3,1.44)-- (-1.7599882128680713,1.7199772653319878);
\draw [line width=2.pt] (-2.946520181237938,2.686988815585338)-- (4.14,6.48);

\draw [fill] (-2.3,1.44) circle (1.5pt);
\draw (-2.78,1.63) node {$x$};
\draw[color=qqqqff] (4.6,1.51) node {$K$};
\draw [fill=ttqqqq] (-1.7599882128680713,1.7199772653319878) circle (1.5pt);
\draw [fill=ttqqqq] (4.886536348815252,5.165968288278606) circle (1.5pt);
\draw[color=black] (0.88,11.81) node {$e$};
\draw[color=black] (-1.7,1.4) node {$\frac{1}{C}$};
\draw[color=black] (0.8,4.2) node {$\frac{1}{c}$};
\draw[color=black] (-0.8,1.4) node {$B_1$};
\draw[color=black] (-3.7,7.5) node {$B_2$};

\end{tikzpicture}
\caption{Illustration of Assumption $(\mathcal{K})$ }
\label{Figure hypothese (K)}
\end{center}
\end{figure}
In this section, we consider the stochastic billiard in such a convex $K$.\\
Let us observe that the case of the disc is a particular case. Moreover, Assumption $(\mathcal{K})$ excludes in particular the case of the polygons: because of the upper bound $C$ on the curvature, the boundary of $K$ can not have "corners", and because of the lower bound $c$, the boundary can not have straight lines.\medskip\\
In the following, $D$ will denote the diameter of $K$, that is
\begin{equation*}
D= \max\{ \lVert x-y \lVert: x,y\in\partial K \}.
\end{equation*}

\subsection{The embedded Markov chain}

\textbf{Notation}: We define $l_{x,y}=\frac{y-x}{\lVert x-y \lVert}=-l_{y,x}$ and we denote by $\varphi_{x,y}$ the angle between $l_{x,y}$ and the normal $n_x$ to $\partial K$ at the point $x$ (see Figure \ref{def de l et phi}). \medskip\\
\begin{figure}
\begin{center}
\definecolor{qqwuqq}{rgb}{0.,0.39215686274509803,0.}
\definecolor{xdxdff}{rgb}{0.49019607843137253,0.49019607843137253,1.}
\begin{tikzpicture}[line cap=round,line join=round,>=triangle 45,x=1.0cm,y=1.0cm]
\clip(-5,-3) rectangle (5,3);
\draw [shift={(-3.4396039331105075,0.5997755434854761)},line width=1.pt] (0,0) -- (-45.82383909301144:0.6) arc (-45.82383909301144:-29.54082006920573:0.6);
\draw [shift={(-0.9879044656034811,-1.923462004330948)},line width=1.pt] (0,0) -- (81.0196404585698:0.6) arc (81.0196404585698:134.17616090698854:0.6) ;
\draw [color=blue,rotate around={0.:(0.,0.)},line width=2.pt] (0.,0.) ellipse (3.6055512754639882cm and 2.cm);
\draw [line width=1.pt] (-3.4396039331105075,0.5997755434854761)-- (-0.9879044656034811,-1.923462004330948);
\draw [dashed,line width=1.pt,domain=-4.3:-2.5] plot(\x,{(--831.9999999999993--220.13465171907242*\x)/124.75331304497902});
\draw [dashed,line width=1.pt,domain=-2.7:1] plot(\x,{(--831.9999999999993--63.22588579862278*\x)/-400.08009690083713});
\draw [dashed,line width=1.pt,domain=-4:-2] plot(\x,{(--297.07060584328354--124.75331304497902*\x)/-220.13465171907242});
\draw [dashed,line width=1.pt,domain=-1.1:-0.7] plot(\x,{(-273.62832530359185-400.08009690083713*\x)/-63.22588579862278});
\draw [->,line width=1.pt] (-3.4396039331105075,0.5997755434854761) -- (-2.7304085424598323,-0.13011344256048274);

\draw[fill] (-0.9879044656034811,-1.923462004330948) circle (2.5pt);
\draw (-0.8,-2.3) node {$x$};
\draw[fill] (-3.4396039331105075,0.5997755434854761) circle (2.5pt);
\draw(-3.5,0.97) node {$y$};
\draw (-2.34,0.47) node {$\varphi_{y,x}$};
\draw (-1.3,-1) node {$\varphi_{x,y} $};
\draw(-3.05,-0.2) node {$l_{y,x}$};
\draw [color=blue] (3.2,1.7) node {$K$};
\end{tikzpicture}
\end{center}
\caption{Definition of the quantities $\varphi_{x,y}$ and $l_{y,x}$ for $x,y\in \partial K$}
\label{def de l et phi}
\end{figure}
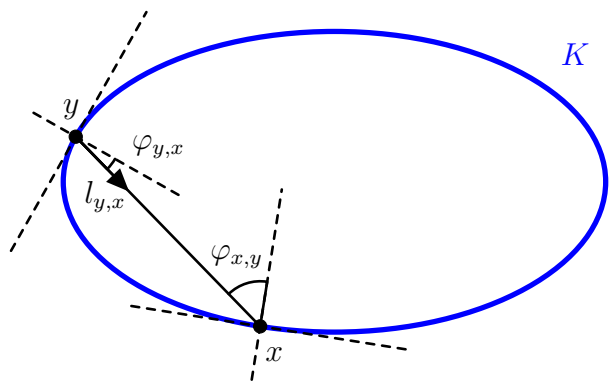
The following property, proved by Comets and al. in \cite{CPSV}, gives the dynamics of the Markov chain $(X_{T_n})_{n\geq 0}$ defined in Section \ref{Subsection Description of the process}

\begin{Prop}\label{Prop noyau de transition}
The transition kernel of the chain $(X_{T_n})_{n\geq 0}$ is given by:
\begin{equation*}
\PP\left( X_{T_{n+1}}\in A \left| X_{T_n}=x \right. \right) = \int_A Q(x,y)\dt y
\end{equation*}
where
\begin{equation*}
Q(x,y)= \frac{\rho(U_x^{-1}l_{x,y})\cos(\varphi_{y,x})}{\lVert x-y \lVert}.
\end{equation*}
\end{Prop}

This proposition is one of the main ingredients to obtain the exponentially-fast convergence of the stochastic billiard Markov chain towards its invariant probability measure.

\begin{Theo}\label{Theo convergence de la CdM dans le convexe}
Let $K\in\RR^2$ satisfying Assumption $(\mathcal{K})$ with diameter $D$. Let $(X_{T_n})_{n\geq 0}$ be the stochastic billiard Markov chain on $\partial K$ verifying Assumption $(\mathcal{H})$.\\
There exists a unique invariant measure $\nu$ on $\partial K$ for $(X_{T_n})_{n\geq 0}$.\\
Moreover, recalling that $\theta^*=\lvert \mathcal{J}\lvert$ in Assumption $(\mathcal{H})$, we have:
\begin{enumerate}
\item if $\theta^*>\frac{C\lvert \partial K \lvert }{8}$, for all $n\geq 0$,
\begin{equation*}
\lVert \PP\left( X_{T_n}\in \cdot\right) - \nu \lVert_{TV}\leq \left(1-q_{\min}\left( \frac{8\theta^*}{C}-\lvert \partial K \lvert \right)\right)^{n-1};
\end{equation*}
\item if $\theta^*\leq\frac{C\lvert \partial K \lvert }{8}$, for all $n\geq 0$ and all $\varepsilon\in \left( 0,\frac{2\theta^*}{C}\right)$,
\begin{equation*}
\lVert \PP\left( X_{T_n}\in \cdot\right) - \nu \lVert_{TV}\leq \left(1-\alpha\right)^{\frac{n}{n_0}-1}
\end{equation*}
where
\begin{equation*}
n_0= \left\lfloor\frac{\frac{\lvert \partial K \lvert}{2}-2\varepsilon}{\frac{4\theta^*}{C}-2\varepsilon} \right\rfloor +1 ~~ \text{and} ~~ \alpha = (\frac{4\theta^*}{C})^{n_0-1}{q_{\min}}^{n_0}\left(4\left( \frac{2n_0\theta^*}{C}-(n_0-1)\varepsilon\right)-\lvert \partial K \lvert\right)
\end{equation*}
\end{enumerate}
with 
\begin{equation*}
q_{\min}=\frac{c\rho_{\min}\cos\left(\frac{\theta^*}{2} \right)}{C D}.
\end{equation*}
\end{Theo}

\begin{proof}
Once more, the existence of the invariant measure is immediate since the state space $\partial K$ of the Markov chain is compact. The following shows its uniqueness and gives the speed of convergence of $(X_{T_n})_{n\geq 0}$ towards $\nu$.\\
Let $(X_{T_n})_{n\geq 0}$ and $(\tilde{X}_{T_n})_{n\geq 0}$ be two versions of the Markov chain with initial conditions $x_0$ and $\tilde{x}_0$ on $\partial K$. In order to have a strictly positive probability to couple $X_{T_n}$ and $\tilde{X}_{T_n}$ at time $n$, it is sufficient that their density functions are bounded from below on an interval of length strictly bigger than $\frac{\lvert\partial K\lvert}{2}$. Let us therefore study the length of set on which $f_{X_{T_n}}$ is bounded from below by a strictly positive constant.\\
Let $x\in \partial K$. For $v\in \SS_x$, we denote by $h_x(v)$ the unique point on $\partial K$ seen from $x$ in the direction of $v$. We firstly get a lower bound on $\lvert h_x(U_x \mathcal{J})\lvert$, the length of the subset of $\partial K$ seen from $x$ with a strictly positive density.\\
It is easy to observe, with a drawing for instance, the following facts:
\begin{itemize}
\item $\lvert h_x(U_x \mathcal{J})\lvert$ increases when $\lVert x-h_x(n_x)\lVert$ increases,
\item $\lvert h_x(U_x \mathcal{J})\lvert$ decreases when the curvature at $h_x(n_x)$ increases,
\item $\lvert h_x(U_x \mathcal{J})\lvert$ decreases when $\lvert \varphi_{h_x(n_x),x}\lvert$ increases.
\end{itemize} 
Therefore, $\lvert h_x(U_x \mathcal{J})\lvert$ is minimal when $\lVert x-h_x(n_x)\lVert$ is minimal, when the curvature at $h_x(n_x)$ is maximal, and then equal to $C$, and finally when $\varphi_{h_x(n_x),x}=0$. Moreover, the minimal value of $\lVert x-h_x(n_x)\lVert$ is $\frac{2}{C}$ since $C$ is the upper bound for the curvature of $\partial K$. The configuration that makes the quantity $\lvert h_x(U_x \mathcal{J})\lvert$ minimal is thus the case where $x$ and $h_x(n_x)$ define a diameter on a circle of diameter $\frac{2}{C}$ (see Figure \ref{Figure taille minimal de h(UJ)}). We immediately deduce a lower bound for $\lvert h_x(U_x \mathcal{J})\lvert$:
\begin{equation*}
\lvert h_x(U_x \mathcal{J})\lvert \geq 2\theta^*\times\frac{2}{C}=\frac{4\theta^*}{C}.
\end{equation*}

\definecolor{qqwuqq}{rgb}{0.,0.39215686274509803,0.}
\definecolor{uuuuuu}{rgb}{0.26666666666666666,0.26666666666666666,0.26666666666666666}
\definecolor{xdxdff}{rgb}{0.49019607843137253,0.49019607843137253,1.}
\begin{figure}
\begin{center}
\begin{tikzpicture}[line cap=round,line join=round,>=triangle 45,x=1.0cm,y=1.0cm, scale=1.2]
\clip(-3,-3) rectangle (3,3);
\draw [shift={(0.,-2.)},line width=1.pt] (0,0) -- (59.69450573316017:0.6) arc (59.69450573316017:120.1275593515289:0.6) -- cycle;
\draw [line width=2.pt] (0.,0.) circle (2.cm);
\draw [line width=1.pt] (0.,-2.)-- (-1.7364862842489186,0.9922778767136672);
\draw [line width=1.pt,] (0.,-2.)-- (1.7426158872198927,0.9814733157905142);
\draw [shift={(0.,0.)},line width=3.pt,color=red]  plot[domain=0.5129350139936567:2.6224465393432705,variable=\t]({1.*2.*cos(\t r)+0.*2.*sin(\t r)},{0.*2.*cos(\t r)+1.*2.*sin(\t r)});
\draw [<->,line width=1.pt] (0.,-2.)-- (0.,2.);

\draw [fill=black] (0.,2.) circle (1pt);
\draw (0.18,2.35) node {$y=h_x(n_x)$};
\draw [fill=black] (0.,-2.) circle (1.0pt);
\draw(0.18,-2.27) node {$x$};
\draw (0.22,-1.2) node {$\theta^*$};
\draw[color=red] (1.8,1.8) node {$h_x(U_x\mathcal{J})$};
\draw[color=black] (0.38,0.17) node {$\frac{2}{C}$};

\end{tikzpicture}
\caption{Worst scenario for the length of $h_x(U_x\mathcal{J})$ }
\label{Figure taille minimal de h(UJ)}
\end{center}
\end{figure}
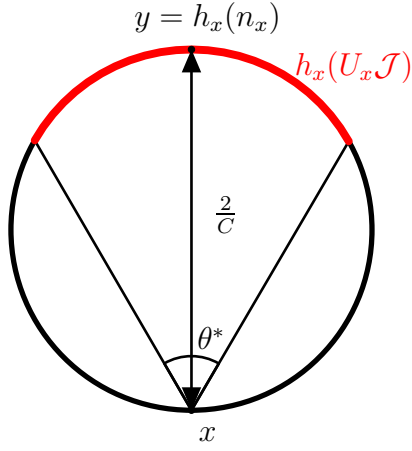

This means that the density function $f_{X_{T_1}}$ of $X_{T_1}$ is strictly positive on a subset of $\partial K$ of length at least $\frac{4\theta^*}{C}$.\\
Let now $\varepsilon\in\left( 0,\frac{2\theta^*}{C}\right)$. As it has been done in Section \ref{Section Stochastic billiard in the disc} for the disc, we can deduce that for all $n\geq 2$, the density function $f_{X_{T_n}}$ is strictly positive on a set of length at least $2n\theta^*\frac{2}{C} - 2(n-1)\varepsilon=\frac{4n\theta^*}{C}- 2(n-1)\varepsilon$.\\
Let define, for $x\in \partial K$ and $n\geq 1$, $\mathcal{J}_{x}^{n}$ the set of points of $\partial K$ that can be reached from $x$ in $n$ bounces by picking for each bounce a velocity in $\mathcal{J}$.\\
We now separate the cases where we can couple in one jump, and where we need more jumps.
\begin{itemize}
\item Case 1: $\theta^*>\frac{C\lvert \partial K \lvert}{8}$. In that case we have, for all $x\in\partial K$, $\lvert \mathcal{J}^1_x \lvert\geq \frac{4\theta^*}{C}>\frac{\lvert \partial K \lvert}{2}$, and we can thus construct a coupling $(X_{T_1},\tilde{X}_{T_1})$ such that:
\begin{equation*}
\PP\left( X_{T_1}=\tilde{X}_{T_1} \right) \geq q_{\min} \left\lvert \mathcal{J}^1_{x_0}\cap \tilde{\mathcal{J}}^1_{\tilde{x}_0} \right\lvert \geq q_{\min}\times 2\left( \frac{4\theta^*}{C} - \frac{\lvert \partial K \lvert}{2}\right)=q_{\min}\left( \frac{8\theta^*}{C}-\lvert \partial K \lvert \right),
\end{equation*}
where $q_{\min}$ is a uniform lower bound of $Q(a,b)$ with $a\in \partial K$ and $b\in h_a(U_a\mathcal{J})$, i.e.
\begin{equation*}
q_{\min}\leq \min_{a\in \partial K,b\in h_a(U_a\mathcal{J})} Q(a,b).
\end{equation*}
Let thus give an explicit expression for $q_{\min}$. Let $a\in \partial K$ and $b\in h_a(U_a\mathcal{J})$. We have
\begin{equation*}
Q(a,b)\geq \frac{\rho_{\min}\cos\left(\varphi_{b,a}\right)}{D}.
\end{equation*}
We could have $\cos\left( \varphi_{b,a}\right)=0$ if $a$ and $b$ were on a straight part of $\partial K$, which is not possible since the curvature of $K$ is bounded from below by $c$. Thus, the quantity $\cos\left( \varphi_{b,a}\right)$ is minimal when $a$ and $b$ are on a part of a disc with curvature $c$. In that case, $\cos\left( \varphi_{b,a}\right)= \frac{\delta c}{2}$, where $\delta$ is the distance between $a$ and $b$ (see the first picture of Figure \ref{Figure qmin}). Since $b\in h_a(U_a\mathcal{J})$, we have $\delta \geq \delta_{\min}:= \frac{2\cos\left( \frac{\theta^*}{2}\right)}{C}$ (see the second picture of Figure \ref{Figure qmin}). Finally we get 
\begin{equation*}
Q(a,b) \geq \frac{c\rho_{\min}\cos\left(\frac{\theta^*}{2} \right)}{C D}=:q_{\min}.
\end{equation*}

\definecolor{qqwuqq}{rgb}{0.,0.39215686274509803,0.}
\definecolor{xdxdff}{rgb}{0.49019607843137253,0.49019607843137253,1.}
\definecolor{ududff}{rgb}{0.30196078431372547,0.30196078431372547,1.}

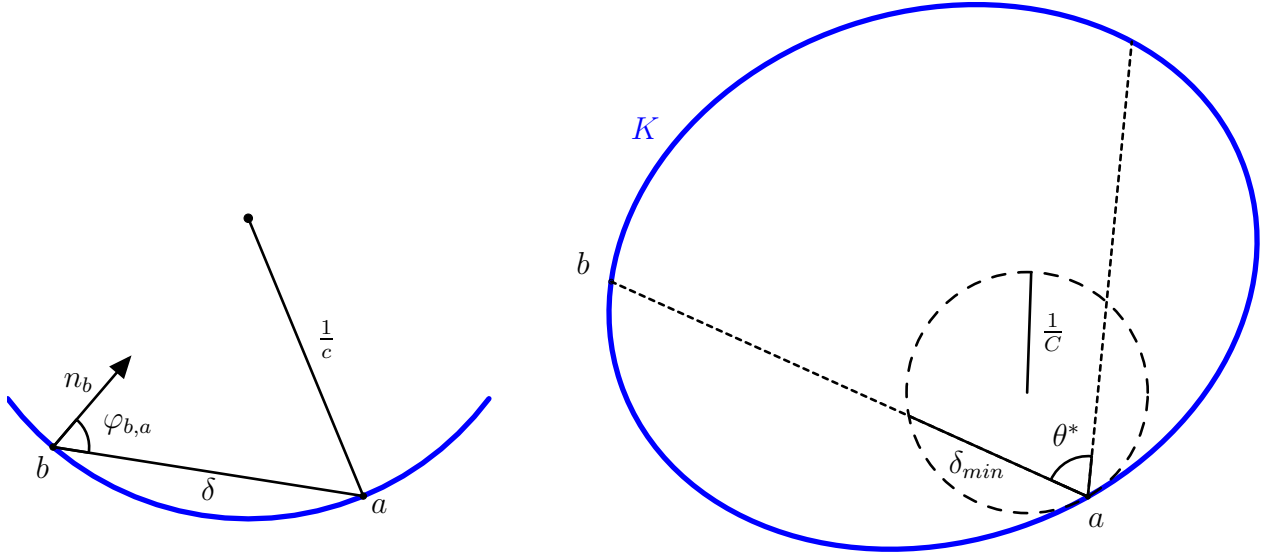
\begin{figure}
\begin{center}
\subfloat{
\begin{tikzpicture}[line cap=round,line join=round,>=triangle 45,x=1.0cm,y=1.0cm,scale=0.8]
\clip(-0.,-2) rectangle (8.7,4.5);
\draw [shift={(0.7559488947515982,0.19524870372101333)},line width=1.pt] (0,0) -- (-8.986261596885749:0.6) arc (-8.986261596885749:49.54804240912539:0.6) -- cycle;
\draw [shift={(4.,4.)},line width=2.pt,color=blue]  plot[domain=3.7850937623830774:5.639684198386302,variable=\t]({1.*5.*cos(\t r)+0.*5.*sin(\t r)},{0.*5.*cos(\t r)+1.*5.*sin(\t r)});
\draw [line width=1.pt] (5.911758963631498,-0.620084162109455)-- (4.,4.);
\draw [->,line width=1.pt] (0.7559488947515982,0.19524870372101333) -- (2.06,1.72);
\draw [line width=1.pt] (0.7559488947515982,0.19524870372101333)-- (5.911758963631498,-0.620084162109455);

\draw [fill] (4.,4.) circle (2pt);
\draw [fill] (0.7559488947515982,0.19524870372101333) circle (1.5pt);
\draw (0.58,-0.13) node {$b$};
\draw [fill] (5.911758963631498,-0.620084162109455) circle (1.5pt);
\draw (6.18,-0.79) node {$a$};
\draw[color=black] (5.3,1.99) node {$\frac{1}{c}$};
\draw[color=black] (1.18,1.27) node {$n_b$};
\draw[color=black] (3.34,-0.5) node {$\delta$};
\draw (2.,0.55) node {$\varphi_{b,a}$};
\end{tikzpicture}
}
\subfloat{
\definecolor{qqwuqq}{rgb}{0.,0.39215686274509803,0.}
\definecolor{xdxdff}{rgb}{0.49019607843137253,0.49019607843137253,1.}
\begin{tikzpicture}[line cap=round,line join=round,>=triangle 45,x=1.0cm,y=1.0cm,scale=0.9]
\clip(-3,-3) rectangle (7.5,6);
\draw [shift={(4.616838632110627,-1.7790238656083985)},line width=1.pt] (0,0) -- (84.48919866102547:0.6) arc (84.48919866102547:155.69483416600787:0.6) -- cycle;

\draw [rotate around={18.68460205564692:(2.33,1.47)},line width=2.pt,color=blue] (2.33,1.47) ellipse (4.874832417498782cm and 3.9161959985037806cm);
\draw [line width=1.pt,dash pattern=on 5pt off 5pt] (3.72,-0.24) circle (1.7812675236915647cm);
\draw [line width=1.pt,dash pattern=on 2pt off 2pt] (4.616838632110627,-1.7790238656083985)-- (-2.423201068148151,1.4004401129060973);
\draw [line width=1.pt,dash pattern=on 2pt off 2pt] (4.616838632110627,-1.7790238656083985)-- (5.265885352937779,4.948290664062741);
\draw [line width=1.pt] (3.72,-0.24)-- (3.780008643680274,1.5402564291814629);
\draw [line width=1.pt] (4.616838632110627,-1.7790238656083985)-- (1.9724089424144924,-0.5847310348742095);

\draw[color=blue] (-1.92,3.67) node {$K$};
\draw [fill=] (4.616838632110627,-1.7790238656083985) circle (1pt);
\draw (4.74,-2.13) node {$a$};
\draw  (-2.423201068148151,1.4004401129060973) circle (1pt);
\draw (-2.84,1.67) node {$b$};
\draw (4.3,-0.87) node {$\theta^*$};
\draw[color=black] (4.1,0.65) node {$\frac{1}{C}$};
\draw[color=black] (2.98,-1.3) node {$\delta_{min}$};

\end{tikzpicture}
}
\caption{Illustration for the calculation of a lower bound for $\cos\left( \varphi_{b,a}\right)$ with $a\in \partial K$ and $b\in h_a(U_a\mathcal{J})$}\label{Figure qmin}
\end{center}
\end{figure}

\item Case 2: $\theta^*\leq \frac{C\lvert \partial K \lvert}{8}$. In that case, we need more than one jump to couple the two Markov chains. Therefore, defining
\begin{equation*}
n_0=\min\left\{n\geq 2: \frac{4n\theta^*}{C}-2(n-1)\varepsilon > \frac{\partial K}{2}  \right\} =\left\lfloor \frac{\frac{\lvert \partial K \lvert}{2}-2\varepsilon}{\frac{4\theta^*}{C}-2\varepsilon} \right\rfloor +1,
\end{equation*}
we get that the intersection $\mathcal{J}^{n_0}_{x_0}\cap \tilde{\mathcal{J}}^{n_0}_{\tilde{x}_0} $ is non-empty, and then we can construct $X_{T_{n_0}}$ and $\tilde{X}_{T_{n_0}}$ such that the probability $\PP\left( X_{T_{n_0}}=\tilde{X}_{T_{n_0}} \right)$ is strictly positive. It remains to estimate a lower bound of this probability.\\
First, we have
\begin{equation*}
\left\lvert \mathcal{J}_{x_0}^{n_0} \cap \tilde{\mathcal{J}}_{\tilde{x}_0}^{n_0} \right\lvert \geq 2\left( \frac{4n_0\theta^*}{C}-2(n_0-1)\varepsilon-\frac{\lvert \partial K \lvert}{2}\right)=4\left( \frac{2n_0\theta^*}{C}-(n_0-1)\varepsilon\right)-\lvert \partial K \lvert.
\end{equation*}
Moreover, let $x\in \{x_0,\tilde{x}_0\}$ and $y\in \mathcal{J}_{x_0}^{n_0} \cap \tilde{\mathcal{J}}_{\tilde{x}_0}^{n_0}$. We have:
\begin{align*}
Q^{n_0}&(x,y)\\
&\geq \int_{h_x(U_x\mathcal{J})} \int_{h_{z_1}(U_{z_1}\mathcal{J})} \cdots \int_{h_{z_{n-2}}(U_{z_{n-2}}\mathcal{J})} Q(x,z_1)Q(z_1,z_2) \cdots Q(z_{n_0-1},y) \dt z_1 \dt z_2 \cdots \dt z_{n_0-1}\\
&\geq (\frac{4\theta^*}{C})^{n_0-1}{q_{\min}}^{n_0}.
\end{align*}
We thus deduce:
\begin{align*}
\PP\left( X_{T_{n_0}}=\tilde{X}_{T_{n_0}} \right) &\geq (\frac{4\theta^*}{C})^{n_0-1}{q_{\min}}^{n_0} \left\lvert \mathcal{J}_{x_0}^{n_0} \cap \tilde{\mathcal{J}}_{\tilde{x}_0}^{n_0} \right\lvert\\
&\geq (\frac{4\theta^*}{C})^{n_0-1}{q_{\min}}^{n_0}\left(4\left( \frac{2n_0\theta^*}{C}-(n_0-1)\varepsilon\right)-\lvert \partial K \lvert\right).
\end{align*}

\end{itemize}
We can now conclude, including the two cases: let define
\begin{equation*}
m_0= \mathbf{1}_{\theta^*>\frac{C\lvert \partial K \lvert }{8}} + \left(\left\lfloor\frac{\frac{\lvert \partial K \lvert}{2}-2\varepsilon}{\frac{4\theta^*}{C}-2\varepsilon} \right\rfloor +1\right) \mathbf{1}_{\theta^*\leq\frac{C\lvert \partial K \lvert }{8}}
\end{equation*}
and
\begin{equation*}
\alpha = q_{\min}\left( \frac{8\theta^*}{C}-\lvert \partial K \lvert \right)\mathbf{1}_{\theta^*>\frac{C\lvert \partial K \lvert }{8}} + (\frac{4\theta^*}{C})^{m_0-1}{q_{\min}}^{m_0}\left(4\left( \frac{2m_0\theta^*}{C}-(m_0-1)\varepsilon\right)-\lvert \partial K \lvert\right)\mathbf{1}_{\theta^*\leq\frac{C\lvert \partial K \lvert }{8}}.
\end{equation*}
We have proved that we can construct a coupling $\left( X_{T_{m_0}},\tilde{X}_{T_{m_0}}\right)$ such that $\PP\left( X_{T_{m_0}}=\tilde{X}_{T_{m_0}} \right) \geq \alpha$, and then we get
\begin{equation*}
\lVert \PP\left( X_{T_n}\in \cdot\right) - \nu \lVert_{TV}\leq \left(1-\alpha\right)^{\frac{n}{m_0}-1}.
\end{equation*}
\end{proof}

\subsection{The continuous-time process}

In this section, we suppose $\lvert \mathcal{J}\lvert =\theta^*=\pi$.

\begin{Prop}\label{Prop T1 dans le convexe}
Let $K\subset \RR^2$ satisfying Assumption $(\mathcal{K})$.
Let $(X_t,V_t)_{t\geq 0}$ the stochastic billiard process evolving in $K$ and verifying Assumption $(\mathcal{H})$ with $\lvert \mathcal{J}\lvert=\pi$.\\
For all $x\in\partial K$, the first hitting-time $T_1^x$ of $\partial K$ starting at point $x$ is $c\rho_{\min}$-continuous on $\left[0,\frac{2}{C}\right]$.
\end{Prop}

\begin{proof}
Let $x\in\partial K$.
Let us recall that the curvature of $K$ is bounded from above by $C$, which means that for each $x\in\partial K$, there is a ball $B_1$ with radius $\frac{1}{C}$ included in $K$ so that the tangent planes of $K$ and $B_1$ at $x$ coincide. Therefore, starting from $x$, the maximal time to go on another point of $\partial K$ is bigger than $\frac{2}{C}$ (the diameter of the ball $B_1$).\\
That is why we are going to prove the continuity of $T_1^x$ on the interval $\left[0,\frac{2}{C}\right]$. Let thus $0\leq r \leq R \leq \frac{2}{C}$.\\
Let $\Theta$ be a random variable living in $\left[-\frac{\pi}{2},\frac{\pi}{2}\right]$ such that the velocity vector $\left( \cos(\Theta),\sin(\Theta)\right)$ follows the law $\gamma$.\\
The time $T^x_1$ being completely determined by the velocity $V_0$ and thus by its angle with respect to $n_x$, it is clear that there exist $-\frac{\pi}{2}\leq \theta_1\leq \theta_2\leq \theta_3\leq \theta_4\leq \frac{\pi}{2}$ such that we have:
\begin{align*}
\PP\left( T_1^x\in \left[r,R\right]\right) &= \PP\left( \Theta\in \left[\theta_1,\theta_2\right]\cup \left[\theta_3,\theta_4\right]\right).
\end{align*}
Then, thanks to assumption $(\mathcal{H})$ on the law $\gamma$, and since we assume here that $\lvert \mathcal{J}\lvert=\pi$, the density function of $\Theta$ is bounded from below by $\rho_{\min}$ on $\left[-\frac{\pi}{2},\frac{\pi}{2}\right]$. It gives:
\begin{align*}
\PP\left( T_1^x\in \left[r,R\right]\right) &\geq \rho_{\min}\left(\theta_2-\theta_1+\theta_4-\theta_3 \right).
\end{align*}
Moreover, since the curvature is bounded from below by $c$, there exists a ball $B_2$ with radius $\frac{1}{c}$ containing $K$ so that the tangent planes of $K$ and $B_2$ at $x$ coincide. And it is easy to see that the differences $\theta_2-\theta_1$ and $\theta_4-\theta_2$ are larger than the difference $\alpha_2-\alpha_1$ where $\alpha_1$ and $\alpha_2$ are the angles corresponding to the distances $r$ and $R$ starting from $x$ and to arrive on the ball $B_2$.\\
The time of hitting the boundary of $B_2$ is equal to $d\in\left[0,\frac{2}{C}\right]$ if the angle between $n_x$ and the velocity is equal to $\arccos\left(\frac{cd}{2}\right)$.  We thus deduce:
\begin{align*}
\PP\left( T_1^x\in \left[r,R\right]\right) &\geq 2\rho_{\min}\left(\arccos\left( \frac{cr}{2}\right)-\arccos\left(\frac{cR}{2} \right)\right)\\
&\geq 2\rho_{\min}\left\lvert \frac{cr}{2}-\frac{cR}{2} \right\lvert\\
&= \rho_{\min}c\left(R-r\right),
\end{align*}
where we have used the mean value theorem for the second inequality.\\
\end{proof}

Let us introduce some constants that will appear in the following results.\\
Let $\beta>0$ and $\delta>0$ such that $\frac{\lvert \partial K \lvert}{3}-\max\{ 2\delta;\beta+\delta\}>0$.\\
Let $\varepsilon\in\left( 0,\min\{\beta;\frac{2}{C}\} \right)$ such that $h >0$ where
\begin{equation}\label{def de h dans le convexe}
h= \frac{\delta}{D} \left( \frac{\beta c}{2}\right)^2 - \varepsilon M,
\end{equation}
with
\begin{equation}\label{def de M dans le convexe}
M=2\left( \frac{1}{\frac{1}{C}-\varepsilon}+\frac{1}{\beta-\varepsilon} +C\right).
\end{equation}

Let us remark that $M$ is non decreasing with $\varepsilon$, so that it is possible to take $\varepsilon$ small enough to have $h>0$.

\begin{Prop}\label{Prop continuite de (X,T) dans convexe}
Let $K\subset \RR^2$ satisfying Assumption $(\mathcal{K})$ with diameter $D$. 
Let $(X_t,V_t)_{t\geq 0}$ the stochastic billiard process evolving in $K$ and verifying Assumption $(\mathcal{H})$ with $\lvert \mathcal{J}\lvert=\pi$.\\
Let $x,\tilde{x}\in \partial K$ with $x\neq \tilde{x}$.\\
There exist $R_1>0$, $R_2>0$ and $J^*\subset \partial K$, with $\lvert J^*\lvert <h\varepsilon$, such that the couples $(X_{T_2}^x,T_2^x)$ and $(\tilde{X}_{\tilde{T}_2}^{\tilde{x}},\tilde{T}_2^{\tilde{x}})$ are both $\eta$-continuous on $J^*\times \left( R_1,R_2\right)$, with
\begin{equation*}
\eta = \frac{1}{2}\left(\frac{c\rho_{\min}}{2D} \right)^2\left(\frac{1}{C}-\varepsilon\right)\left(\beta-\varepsilon\right).
\end{equation*}
Moreover we have $R_2-R_1\geq 2\left( h\varepsilon - \lvert J^*\lvert\right)$.
\end{Prop}

\begin{Rema}
The following proof is largely inspired by the proof of Lemma 5.1 in \cite{CPSV}.
\end{Rema}

\begin{proof}
Let $x,\tilde{x}\in \partial K$, $x\neq \tilde{x}$. Let us denotes by $\Delta_{x\tilde{x}}$ the bisector of the segment defined by the two points $x$ and $\tilde{x}$. The intersection $\Delta_{x\tilde{x}}\cap \partial K$ contains two points, let thus define $\bar{y}$ the one which achieves the larger distance towards $x$ and $\tilde{x}$ (we consider this point of intersection since we need in the sequel to have a lower bound on $\lVert x-\bar{y}\lVert$ and $\lVert \tilde{x}-\bar{y}\lVert$).\\
Let $t\in I\mapsto g(t)$ be a parametrization of $\partial K$ with $g(0)=\bar{y}$, such that $\lVert g'(t) \lVert=1$ for all $t\in I$. Consequently, the length of an arc satisfies length$(g_{|[s,t]})=||g(t)-g(s)||=|t-s|$. We can thus write $I=\left[0, \left| \partial K \right| \right]$, and $g(0)=g(\left| \partial K \right|)$. Note that the parametrization $g$ is $C^2$ thanks to Assumption $(\mathcal{K})$.\\
In the sequel, for $z\in \partial K$, we write $s_z$ (or $t_z$) for the unique $s\in I$ such that $g(s)=z$. And for $A\subset \partial K$, we define $I_A=\{ t\in I: g(t)\in A\}$.\\
Let define, for $s,t\in I$ and $w\in\{x,\tilde{x}\}$:
\begin{equation*}
\varphi_{w}(s,t)= \lVert w-g(s)\lVert + \lVert g(s)-g(t)\lVert.
\end{equation*}

\begin{Lemm}\label{Lemme minoration phi'(s)}
There exists an interval $I^*_{\beta,\delta}\subset I$, satisfying $\lvert I^*_{\beta,\delta}\lvert <h\varepsilon$, such that for $w\in\{x,\tilde{x}\}$:
 \begin{equation*}
\lvert \partial_s \varphi_{w}(s,t)\lvert  \geq h, ~~~ \text{for } s\in B_{\bar{y}}^\varepsilon ~ \text{ and }~t\in I^*_{\beta,\delta},
\end{equation*}
where $B_{\bar{y}}^\varepsilon= \{ s\in I; \lvert s-s_{\bar{y}} \lvert \leq \varepsilon\}$.
\end{Lemm}
We admit this lemma for the moment and prove it after the end of the current proof.\\
Let suppose for instance that $\partial_s \varphi_{w}(s,t)$ is positive for $s\in B_{\bar{y}}^\varepsilon$ and $t\in I^*_{\beta,\delta}$, for $w=x$ and $w=\tilde{x}$. If one or both of $\partial_s \varphi_{x}(s,t)$ and $\partial_s \varphi_{\tilde{x}}(s,t)$ are negative, we just need to consider $\lvert \varphi_x\lvert$ or $\lvert \varphi_{\tilde{x}}\lvert$, and everything works similarly.\\
We thus have, by the lemma:
 \begin{equation*}
\partial_s \varphi_{w}(s,t)  \geq h, ~~~ \text{for } s\in B_{\bar{y}}^\varepsilon ~ \text{ and }~t\in I^*_{\beta,\delta}.
\end{equation*}
Let us now define:
\begin{equation*}
r_1= \sup_{t\in I^*_{\beta,\delta} } \inf_{s\in B_{\bar{y}}^{\varepsilon}} \varphi_{x}(s,t) ~~~~ \text{and} ~~~~ r_2= \inf_{t\in I^*_{\beta,\delta}}  \sup_{s\in B_{\bar{y}}^{\varepsilon}} \varphi_{x}(s,t)
\end{equation*}
and 
\begin{equation*}
\tilde{r_1}= \sup_{t\in I^*_{\beta,\delta} } \inf_{s\in B_{\bar{y}}^{\varepsilon}} \varphi_{\tilde{x}}(s,t) ~~~~ \text{and} ~~~~ \tilde{r_2}= \inf_{t\in I^*_{\beta,\delta}}  \sup_{s\in B_{\bar{y}}^{\varepsilon}} \varphi_{\tilde{x}}(s,t).
\end{equation*}
Since $s\mapsto \varphi_{x}(s,t)$ and $s\mapsto \varphi_{\tilde{x}}(s,t)$ are strictly increasing on $B_{\bar{y}}^{\varepsilon}$ for all $t\in I^*_{\beta,\delta} $, we deduce that, considering $B_{\bar{y}}^{\varepsilon}$ as the interval $(s_1,s_2)$,
\[ r_1= \sup_{t\in I^*_{\beta,\delta} }\varphi_{x}(s_1,t), ~~ r_2= \inf_{t\in I^*_{\beta,\delta} }\varphi_{x}(s_2,t), ~~ \tilde{r_1}= \sup_{t\in I^*_{\beta,\delta} } \varphi_{\tilde{x}}(s_1,t), ~~ \tilde{r_2}= \inf_{t\in I^*_{\beta,\delta} } \varphi_{\tilde{x}}(s_2,t). \]
\begin{Lemm}\label{Lemme r1<r2}
We have $r_1<r_2$ and $\tilde{r_1}<\tilde{r_2}$.\\
Moreover, there exist $R_1,R_2$ with $0\leq R_1<R_2$ satisfying $R_2-R_1\geq 2(h\varepsilon - \lvert I_{\beta,\delta}^*\lvert)$, such that $(r_1,r_2)\cap (\tilde{r_1},\tilde{r_2})=(R_1, R_2)$.
\end{Lemm}
We admit this result to continue the proof, and will give a demonstration later.\\
We can now prove that the pairs $\left(X_{T_2}^x,T_2^x \right)$ and $\left(\tilde{X}_{\tilde{T}_2}^{\tilde{x}},\tilde{T}_2^{\tilde{x}} \right)$ are both $\eta$-continuous on $I^*_{\beta,\delta}\times (R_1,R_2)$ with some $\eta>0$ that we are going to define after the computations.\\
We first prove that $\left(X_{T_2}^x,T_2^x \right)$ is $\eta$-continuous on $I^*_{\beta,\delta}\times (r_1,r_2)$. By the same way we can prove that $\left(\tilde{X}_{\tilde{T}_2}^{\tilde{x}},\tilde{T}_2^{\tilde{x}} \right)$ is $\eta$-continuous on $I^*_{\beta,\delta}\times (\tilde{r}_1,\tilde{r}_2)$. These two facts imply immediately the continuity with $(R_1,R_2)$ since the interval $(R_1,R_2)$ is included in $(r_1,r_2)$ and $(\tilde{r}_1,\tilde{r}_2)$.\\
Let $(u_1,u_2)\subset (r_1,r_2)$ and $A\subset I^*_{\beta,\delta}$. We have:
\begin{align*}
\PP\left( X_{T_2}^x \in A, T_2^x \in (u_1,u_2) \right) &\geq \int_{I_A} \int_{B_{\bar{y}}^{\varepsilon}} Q(x,g(s))Q(g(s),g(t)) \mathbf{1}_{\varphi_{x}(s,t)\in (u_1,u_2)} \dt s \dt t.
\end{align*}
Let $s\in B_{\bar{y}}^{\varepsilon}$ and $t\in I^*_{\beta,\delta}$. We now give a lower bound of $Q(x,g(s))$ and $Q(g(s),g(t))$.\\
Proposition \ref{Prop noyau de transition} gives:
\begin{align*}
Q(x,g(s)) &= \frac{\rho(U_x^{-1}l_{x,g(s)})\cos\left(\varphi_{g(s),x}\right)}{\lVert x-g(s) \lVert}\\
&\geq \frac{c\rho_{\min}}{2D}\left( \frac{1}{C}-\varepsilon\right),
\end{align*}
where we have used the same method as in he proof of Theorem \ref{Theo convergence de la CdM dans le convexe} (with Figure \ref{Figure qmin}) to get that $\cos\left(\varphi_{g(s),x}\right)\geq \frac{\lVert x-g(s)\lVert c}{2}$, and then the fact that $\lVert x-g(s)\lVert\geq \frac{1}{C}-\varepsilon$. Let us prove this latter. With the notations of Figure \ref{Figure distance x ybarre}, by Pythagore's theorem we have, for $\bar{y}\in \{ \bar{y}_1,\bar{y}_2\}$, $\lVert x-\bar{y}\lVert ^2= \left(\frac{\lVert x -\tilde{x}\lVert}{2}\right)^2 + \lVert u-\bar{y}\lVert ^2$. Moreover, since the curvature of $K$ is bounded by $C$, it follows that $\lVert \bar{y}_1 - \bar{y}_2\lVert \geq \frac{2}{C}$, and then $\max\{ \lVert u-\bar{y}_1\lVert; \lVert u-\bar{y}_2\lVert \}\geq \frac{1}{C}$. We deduce: $\max\{ \lVert x-\bar{y}_1\lVert; \lVert x-\bar{y}_2\lVert\}\geq \frac{1}{C}$. Therefore, by the definition of $\bar{y}$, we have $\lVert x - \bar{y}\lVert \geq \frac{1}{C}$. Thus, the reverse triangle inequality gives, for $s\in B_{\bar{y}}^{\varepsilon}$, $\lVert x-g(s)\lVert\geq \frac{1}{C}-\varepsilon$.\\
\begin{figure}
\begin{center}
\definecolor{uuuuuu}{rgb}{0.26666666666666666,0.26666666666666666,0.26666666666666666}
\definecolor{qqqqff}{rgb}{0.,0.,1.}
\begin{tikzpicture}[line cap=round,line join=round,>=triangle 45,x=1.0cm,y=1.0cm,scale=1.3]
\clip(-3.984313990093323,-3) rectangle (5.773372930552262,2.5);
\draw [rotate around={14.083870871348129:(0.41,-0.35)},line width=1.pt,color=qqqqff] (0.41,-0.35) ellipse (3.4720730849085863cm and 1.8820976348071405cm);
\draw [line width=0.8pt,dash pattern=on 3pt off 3pt] (-2.2994863080744112,0.33944225108316034)-- (2.303857232870347,1.5759847330482368);
\draw [line width=0.8pt,dash pattern=on 3pt off 3pt,domain=-3.984313990093323:5.773372930552262] plot(\x,{(-1.1943138527040222--4.603343540944758*\x)/-1.2365424819650765});
\draw [line width=1.pt] (-2.2994863080744112,0.33944225108316034)-- (0.8453447444896267,-2.181161145735769);
\draw [line width=1.pt] (0.8453447444896267,-2.181161145735769)-- (2.303857232870347,1.5759847330482368);
\draw [line width=1.pt] (2.303857232870347,1.5759847330482368)-- (-0.13114446900544113,1.454068035069442);
\draw [line width=1.pt] (-2.2994863080744112,0.33944225108316034)-- (-0.13114446900544113,1.454068035069442);

\draw[color=qqqqff] (-1.47438874171917,1.2909333816423447) node {$K$};
\draw [fill=black] (-2.2994863080744112,0.33944225108316034) circle (1pt);
\draw[color=black] (-2.6306464404083862,0.5576967922296723) node {$x$};
\draw [fill=black] (2.303857232870347,1.5759847330482368) circle (1pt);
\draw[color=black] (2.445606870910126,1.7844580091316433) node {$\tilde{x}$};
\draw[color=black] (0,2.2) node {$\Delta_{x,\tilde{x}}$};
\draw [fill=uuuuuu] (0.8453447444896267,-2.181161145735769) circle (1.0pt);
\draw[color=uuuuuu] (1.1483421357953951,-2.3752495654210177) node {$\bar{y}_2$};
\draw [fill=uuuuuu] (-0.13114446900544113,1.454068035069442) circle (1.0pt);
\draw[color=uuuuuu] (-0.03611697017892446,1.7421558982039893) node {$\bar{y}_1$};
\draw[color=uuuuuu] (0.25,0.8) node {$u$};

\end{tikzpicture}
\end{center}
\caption{Upper bound for the distance $\lVert w-\bar{y}\lVert$, $w\in\{ x,\tilde{x}\}$}
\label{Figure distance x ybarre}
\end{figure}
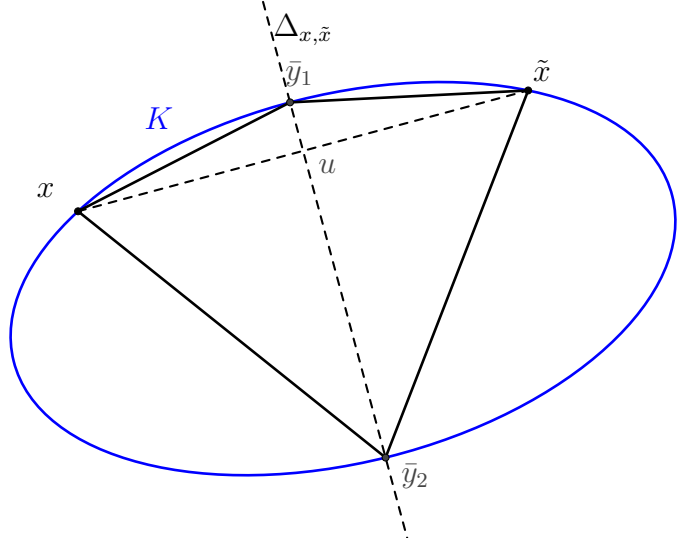
By the same way, since $\lVert g(t)-g(s)\lVert \geq \beta-\varepsilon$, we have:
\begin{align*}
Q\left( g(s),g(t) \right) \geq \frac{c\rho_{\min}}{2D}\left( \beta-\varepsilon\right).
\end{align*}
Therefore we get:
\begin{equation*}
\PP\left( X_{T_2}^x \in A, T_2^x \in (u_1,u_2) \right)  \geq a \int_{I_A} \int_{ B_{\bar{y}}^{\varepsilon}} \mathbf{1}_{\varphi_{x}(s,t)\in (u_1,u_2)} \dt s \dt t,
\end{equation*}
with 
\begin{equation}\label{def a convexe}
a=\left(\frac{c\rho_{\min}}{2D} \right)^2\left(\frac{1}{C}-\varepsilon\right)\left(\beta-\varepsilon\right).
\end{equation}
Let define, for $t\in I^*_{\beta,\delta}$:
\[ M_{x,t}(u_1,u_2):=\left\{ s\in B_{\bar{y}}^{\varepsilon} : \varphi_{x}(s,t)\in (u_1,u_2) \right\}. \]
Using the fact that $s\mapsto\varphi_{x}(s,t)$ is strictly increasing on $B_{\bar{y}}^{\varepsilon}$ for $t\in I^*_{\beta,\delta}$ we get ($\varphi_w^{-1}(s,t)$ stands for the inverse function of $s\mapsto \varphi_x(s,t)$):
\begin{align*}
\left\lvert  M_{x,t}(u_1,u_2) \right\lvert &= \left\lvert \left\{ s\in B_{\bar{y}}^{\varepsilon} : s\in \left( \varphi_{x}^{-1}(u_1,t), \varphi_{x}^{-1}(u_2,t) \right)\right\} \right\lvert\\
&= \left\lvert \left( s_1,s_2 \right)\cap \left( \varphi_{x}^{-1}(u_1,t), \varphi_{x}^{-1}(u_2,t) \right)  \right\lvert .
\end{align*}
By definition of $r_1$ and $r_2$, and since $(u_1,u_2)\subset (r_1,r_2)$ we have:
\[ \varphi_{x}(s_1,t)\leq r_1 \leq u_1 ~~~~ \text{and} ~~~~ \varphi_{x}(s_2,t) \geq r_2\geq u_2, \]
and since $s\mapsto\varphi_{x}(s,t)$ is strictly increasing:
\[ s_1\leq \varphi_{x}^{-1}(u_1,t) ~~~~\text{and} ~~~~  s_2\geq \varphi_{x}^{-1}(u_2,t). \]
Therefore we deduce:
\begin{align*}
\left\lvert M_{x,t}(u_1,u_2) \right\lvert &= \left\lvert \left( \varphi_{x}^{-1}(u_1,t), \varphi_{x}^{-1}(u_2,t) \right) \right\lvert \\
&= \left\lvert \varphi_{x}^{-1}\left( (u_1,u_2),t \right) \right\lvert\\
&\geq \frac{1}{2}(u_2-u_1).
\end{align*}
For the last inequality we have used the following property. Let $\psi:\RR\mapsto \RR$ a function. If for all $x\in [a_1,a_2]$ we have $c_1<\psi'(x)<c_2$ with $0<c_1<c_2<\infty$, then for any interval $I\subset [\psi(a_1),\psi(a_2)]$, we have $c_2^{-1}\lvert I \lvert \leq \lvert \psi^{-1}(I)\lvert \leq c_1^{-1} \lvert I \lvert$. In our case, the Cauchy-Schwarz inequality gives $\partial_s \varphi_x (s,t) \leq 2$ (see Equation \eqref{eq derive phi_s} for the expression of $\partial_s \varphi_x (s,t)$).\\
Finally we get, with $a$ given by \eqref{def a convexe}:
\begin{align*}
\PP\left( X_{T_2}^x \in A, T_2^x \in (u_1,u_2) \right) &\geq a \int_{A} \frac{1}{2}(u_2-u_1) \dt z\\
&= \frac{a}{2}(u_2-u_1)\lvert A \lvert,
\end{align*}
which proves that $\left(X_{T_2}^x , T_2^x\right)$ is $\frac{a}{2}$-continuous on $I^*_{\beta,\delta}\times (\tilde{r}_1,\tilde{r}_2)$.\\
Thanks to the remarks before, the proof is completed with $\eta=\frac{a}{2}$ and $J=I^*_{\beta,\delta}$.\\
\end{proof}

Let us now give the proofs of Lemma \ref{Lemme minoration phi'(s)} and \ref{Lemme r1<r2} that we have admitted so far.

\begin{proof}[Proof of Lemma \ref{Lemme minoration phi'(s)}]
We use the notations introduced at the end of the proof of Proposition \ref{Prop continuite de (X,T) dans convexe}.\\
We have, for $s,t\in I$:
\begin{equation}\label{eq derive phi_s}
\partial_s \varphi_{w}(s,t)=\left\langle \frac{g(s)-w}{\lVert g(s)-w \lVert} +\frac{g(s)-g(t)}{\lVert g(s)-g(t)\lVert} , g'(s)\right\rangle.
\end{equation}
By the definition of $g$, we note that $g'(s)$ is a director vector of the tangent line of $\partial K$ at point $g(s)$.\\
It is easy to verify that for $w\in\{x,\tilde{x}\}$, there exists a unique $t\in I\setminus \{s_{\bar{y}} \}  $ such that 
\begin{equation}\label{derive phi x egal zero}
\partial_s\varphi_{w}(s_{\bar{y}},t)=0.
\end{equation} 
For $w=x$ (resp. $w=\tilde{x}$), we denote by $t_{z_x}$ (resp. $t_{z_{\tilde{x}}}$) this unique element of $I$. With our notations we thus have $g(t_{z_x})=z_x$ and $g(t_{z_{\tilde{x}}})=z_{\tilde{x}}$.\\
Let $w\in\{x,\tilde{x}\}$. We have:
\begin{align*}
\partial_t \partial_s \varphi_{w}(s,t) &= \partial_t \left( \left\langle \frac{g(s)-g(t)}{\lVert g(s)-g(t)\lVert} , g'(s)\right\rangle \right)\\
&= \frac{1}{\lVert g(t)-g(s)\lVert}\left(-\left\langle g'(t),g'(s) \right\rangle + \left\langle \frac{g(t)-g(s)}{\lVert g(t)-g(s) \lVert},g'(t) \right\rangle \left\langle\frac{g(t)-g(s)}{\lVert g(t)-g(s) \lVert},g'(s) \right\rangle \right).
\end{align*}
Let us look at the term in parenthesis. Let us denote by $\textbf{[}u,v\textbf{]}$ the oriented angle between the vectors $u,v\in\RR^2$. We have:
\begin{align*}
-\left\langle g'(t),g'(s) \right\rangle &+ \left\langle \frac{g(t)-g(s)}{\lVert g(t)-g(s) \lVert},g'(t) \right\rangle \left\langle\frac{g(t)-g(s)}{\lVert g(t)-g(s) \lVert},g'(s) \right\rangle\\
&= -\cos\left(\textbf{[}g'(t),g'(s)\textbf{]} \right) + \cos\left(\textbf{[}g(t)-g(s),g'(t)\textbf{]} \right)\cos\left(\textbf{[}g(t)-g(s),g'(s)\textbf{]} \right)\\
&= -\cos\left(\textbf{[}g'(t),g'(s)\textbf{]} \right) +\frac{1}{2}\cos\left(\textbf{[}g(t)-g(s),g'(t)\textbf{]}- \textbf{[}g(t)-g(s),g'(s)\textbf{]}\right)\\
&\hspace*{5.5cm} +\frac{1}{2}\cos\left(\textbf{[}g(t)-g(s),g'(t)\textbf{]}+ \textbf{[}g(t)-g(s),g'(s)\textbf{]}\right)\\
&= -\cos\left(\textbf{[}g'(t),g'(s)\textbf{]}\right) + \frac{1}{2}\cos\left(\textbf{[}g'(s),g'(t)\textbf{]} \right)  \\
&\hspace*{5.5cm}+\frac{1}{2}\cos\left(\textbf{[}g(t)-g(s),g'(t)\textbf{]}+ \textbf{[}g(t)-g(s),g'(s)\textbf{]}\right)\\
&= -\frac{1}{2}\cos\left(\textbf{[}g'(t),g'(s)\textbf{]}\right) +\frac{1}{2}\cos\left(\textbf{[}g(t)-g(s),g'(t)\textbf{]}+ \textbf{[}g(t)-g(s),g'(s)\textbf{]}\right)\\
&= -\sin \left(\frac{1}{2}\left(\textbf{[}g(t)-g(s),g'(t)\textbf{]}+ \textbf{[}g(t)-g(s),g'(s)\textbf{]}+ \textbf{[}g'(t),g'(s)\textbf{]}\right)\right)\times\\
&\hspace*{3cm} \sin \left( \frac{1}{2}\left(\textbf{[}g(t)-g(s),g'(t)\textbf{]}+ \textbf{[}g(t)-g(s),g'(s)\textbf{]}- \textbf{[}g'(t),g'(s)\textbf{]}\right) \right)\\
&=-\sin \left( \textbf{[}g(t)-g(s),g'(s)\textbf{]}\right)\sin \left(\textbf{[}g(t)-g(s),g'(t)\textbf{]} \right).
\end{align*}
Therefore we get
\begin{equation*}
\partial_t \partial_s \varphi_{w}(s,t) = -\frac{1}{\lVert g(t)-g(s)\lVert}\sin \left( \textbf{[}g(t)-g(s),g'(s)\textbf{]}\right)\sin \left(\textbf{[}g(t)-g(s),g'(t)\textbf{]} \right),
\end{equation*}
and then
\begin{align*}
\left\lvert \partial_t \partial_s \varphi_{w}(s,t)\right\lvert & = \frac{1}{\lVert g(t)-g(s)\lVert} \left \lvert \sin \left( \textbf{[}g(t)-g(s),g'(s)\textbf{]}\right)\sin \left(\textbf{[}g(t)-g(s),g'(t)\textbf{]} \right)\right\lvert\\
&= \frac{1}{\lVert g(t)-g(s)\lVert} \left \lvert \cos \left( \varphi_{g(s),g(t)}\right)\cos\left(\varphi_{g(t),g(s)} \right)\right\lvert
\end{align*}
Let $t\in I$ such that $\lvert t-s_{\bar{y}} \lvert\geq \beta$ (we recall that $\beta$ is introduced at the beginning of the section). Using once more Figure \ref{Figure qmin}, we get, as we have done in the proof of Theorem \ref{Theo convergence de la CdM dans le convexe}:
\begin{align}\label{minoration derive t,s phi}
\left\lvert \partial_t \partial_s \varphi_{w}(s,t) \right\lvert &\geq \frac{1}{\lVert g(t)-g(s)\lVert} \left(\frac{\beta c}{2} \right)^2 \nonumber\\
&\geq \frac{1}{D}\left(\frac{\beta c}{2} \right)^2.
\end{align}
Using Equations \eqref{derive phi x egal zero} and Equation \eqref{minoration derive t,s phi}, the mean value theorem gives: for $t\in I$ such that $\lvert t-s_{\bar{y}}\lvert \geq \beta$ and $\lvert t-t_{z_w}\lvert \geq \delta$ ($\delta$ is introduced at the beginning of the section),
\begin{equation}\label{minoration derive s phi(ybarre,t)}
\left\lvert \partial_s \varphi_{w}(s_{\bar{y}},t)\right\lvert =\left\lvert \partial_s \varphi_{w}(s_{\bar{y}},t)-\partial_s \varphi_{w}(s_{\bar{y}},t_{z_w})\right\lvert  \geq  \frac{1}{D}\left( \frac{\beta c}{2}\right)^2\lvert t-t_{z_w}\lvert\geq   \frac{\delta}{D}\left( \frac{\beta c}{2}\right)^2.
\end{equation}
We want now such an inequality for $s\in I$ near from $s_{\bar{y}}$. We thus compute:
\begin{align*}
\partial_s^2 \varphi_w(s,t) &= \frac{1}{\lVert w-g(s) \lVert} + \frac{1}{\lVert g(s)-g(t) \lVert} + \left\langle \frac{g(s)-w}{\lVert g(s)-w\lVert}+\frac{g(s)-g(t)}{\lVert g(s)-g(t)\lVert },g^{''}(s) \right\rangle\\
&\hspace*{1cm} -\frac{1}{\lVert w-g(s)\lVert}\left\langle\frac{w-g(s)}{\lVert w-g(s)\lVert},g'(s) \right\rangle^2 -\frac{1}{\lVert g(s)-g(t)\lVert}\left\langle\frac{g(s)-g(t)}{\lVert g(s)-g(t)\lVert},g'(s) \right\rangle^2.
\end{align*}
We immediately deduce, using the Cauchy-Schwarz inequality, and the fact that $\lVert g'(s)\lVert=1$ for all $s\in I$:
\begin{align*}
\lvert \partial_s^2 \varphi_w(s,t) \lvert &\leq  \frac{1}{\lVert w-g(s) \lVert} + \frac{1}{\lVert g(s)-g(t) \lVert} +2 \lVert g^{''}(s)\lVert + \frac{1}{\lVert w-g(s) \lVert} + \frac{1}{\lVert g(s)-g(t) \lVert}  \\
&\leq 2\left(\frac{1}{\lVert w-g(s)\lVert} + \frac{1}{\lVert g(s)-g(t)\lVert} +C\right),
\end{align*}
where we recall that $C$ is the upper bound on the curvature of $K$.\\
Let now $t\in I$ such that $\lvert t-s_{\bar{y}}\lvert \geq \beta$ and $\lvert t-t_{z_w}\lvert \geq \delta$, and let $s\in I$ such that $\lvert s-s_{\bar{y}}\lvert \leq \varepsilon$. With such $s$ and $t$ we have $\lvert t-s\lvert\geq \beta-\varepsilon$. Moreover, we have already seen in proof of Proposition \ref{Prop continuite de (X,T) dans convexe} (with Figure \ref{Figure distance x ybarre}) that $\lVert w- g(s)\lVert \geq \frac{1}{C}-\varepsilon$ for $s\in B^\varepsilon_{\bar{y}}$. Therefore, for such $s$ and $t$:
\begin{equation}\label{derive s,s phi}
\lvert \partial_s^2 \varphi_w(s,t) \lvert \leq 2\left( \frac{1}{\frac{1}{C}-\varepsilon}+\frac{1}{\beta-\varepsilon}+C \right)  = M>0.
\end{equation}
Using once again the mean value theorem with Equations \eqref{minoration derive s phi(ybarre,t)} and \eqref{derive s,s phi}, we deduce that for all $t\in I$ such that $\lvert t-s_{\bar{y}}\lvert \geq \beta$ and $\lvert t-t_{z_w}\lvert \geq \delta$, and for all $s\in I$ such that $\lvert s-s_{\bar{y}}\lvert \leq \varepsilon$:
\begin{equation*}
\left\lvert \partial_s \varphi_{w}(s,t) \right\lvert \geq \frac{\delta}{D}\left( \frac{\beta c}{2}\right)^2 - \varepsilon M =h >0.
\end{equation*}
Let now take $I^*_{\beta,\delta}\subset I\setminus\{s_{\bar{y}},t_{z_x},t_{z_{\tilde{x}}}\}$ an interval of length strictly smaller than $h\varepsilon$ (this condition on the length of $I^*_{\beta,\delta}$ is not necessary for the lemma, but for the continuation of the proof of the proposition), and such that for all $t\in I^*_{\beta,\delta}$, $\lvert t-t_{z_x}\lvert \geq \delta$, $\lvert t-t_{z_{\tilde{x}}}\lvert \geq \delta$ and $\lvert t-s_{\bar{y}}\lvert \geq \beta$. In order to ensure that $I^*_{\beta,\delta}$ is not empty, we take $\beta$ and $\delta$ such that $\frac{\lvert \partial K \lvert}{3}-\max\{2\delta;\beta+\delta\} >0$. Indeed, it is necessary that one of the intervals $"(t_{z_x},t_{z_{\tilde{x}}})"$, $"(t_{z_x},s_{\bar{y}})"$ and $"(s_{\bar{y}},t_{z_{\tilde{x}}})"$ at which we removes a length $\beta$ or $\delta$ on the good extremity, is not empty. And since the larger of these intervals has a length at least $\frac{\partial K}{3}$, we obtain the good condition on $\beta$ and $\delta$.\\
We thus just proved that $\left\lvert \partial_s \varphi_{w}(s,t) \right\lvert \geq h$ for $s\in B_{\bar{y}}^\varepsilon$ and $t\in I^*_{\beta,\delta}$, which is the result of the lemma.
\end{proof}

\begin{proof}[Proof of Lemma \ref{Lemme r1<r2}]
Let first prove that $r_1<r_2$. We do it only for $r_1$ and $r_2$ since it is the same for $\tilde{r_1}$ and $\tilde{r_2}$. We have:
\begin{align*}
r_2-r_1 &= \inf_{t\in I^*_{\beta,\delta} }\varphi_{x}(s_2,t) - \sup_{t\in I^*_{\beta,\delta} }\varphi_{x}(s_1,t)\\
&= \inf_{t\in I^*_{\beta,\delta} }\varphi_{x}(s_2,t) - \inf_{t\in I^*_{\beta,\delta} }\varphi_{x}(s_1,t) - \left( \sup_{t\in I^*_{\beta,\delta} }\varphi_{x}(s_1,t) -\inf_{t\in I^*_{\beta,\delta} }\varphi_{x}(s_1,t)\right)\\
&\geq h(s_2-s_1)- \left(\sup_{t\in I^*_{\beta,\delta}} \left\lvert \partial_t\varphi_{x}(s_1,t) \right\lvert\right) \left\lvert I^*_{\beta,\delta} \right\lvert\\
&\geq 2h\varepsilon - \left\lvert I^*_{\beta,\delta} \right\lvert,
\end{align*}
and this quantity is strictly positive since $\lvert I^*_{\beta,\delta}\lvert <h\varepsilon$ by construction.\\
For the first inequality, we have used the mean value theorem twice, and for the last inequality, we have used the fact that $\sup_{t\in I^*_{\beta,\delta}} \left\lvert \partial_t\varphi_{x}(s_1,t) \right\lvert =  \sup_{t\in I^*_{\beta,\delta}}\left\lvert \left\langle \frac{g(t)-g(s_1)}{\lVert g(t)-g(s_1)\lVert},g'(t)\right\rangle \right\lvert \leq 1$ thanks to the Cauchy-Schwarz inequality.\\
Let us now prove that the intersection $(r_1,r_2)\cap (\tilde{r}_1,\tilde{r}_2)$ is not empty.\\
Let $t\in I^*_{\beta,\delta}$, we have:
\begin{align*}
r_2-\varphi_{x}(s_{\bar{y}},t) &= \inf_{t\in I^*_{\beta,\delta} }\varphi_{x}(s_2,t) - \varphi_{x}(s_{\bar{y}},t) \\
&= \inf_{t\in I^*_{\beta,\delta} }\varphi_{x}(s_2,t) - \inf_{t\in I^*_{\beta,\delta} }\varphi_{x}(s_{\bar{y}},t) - \left( \varphi_{x}(s_{\bar{y}},t) - \inf_{t\in I^*_{\beta,\delta} }\varphi_{x}(s_{\bar{y}},t)\right)\\
&\geq h(s_2-s_{\bar{y}})-\lvert I^*_{\beta,\delta}\lvert\\
&= h\varepsilon - \lvert I^*_{\beta,\delta}\lvert\\
&>0,
\end{align*}
once again thanks to the mean value theorem. Similarly we have 
\begin{align*}
\varphi_{x}(s_{\bar{y}},t) - r_1 &= \varphi_{x}(s_{\bar{y}},t) -\sup_{t\in I^*_{\beta,\delta} }\varphi_{x}(s_1,t)\\
&= \varphi_{x}(s_{\bar{y}},t) - \sup_{t\in I^*_{\beta,\delta}} \varphi_x(s_{\bar{y}},t) - \left( \sup_{t\in I^*_{\beta,\delta} }\varphi_{x}(s_1,t) - \sup_{t\in I^*_{\beta,\delta}} \varphi_x(s_{\bar{y}},t)\right)\\
&\geq -\lvert I^*_{\beta,\delta} \lvert +h(s_{\bar{y}}-s_1)\\
&=  h\varepsilon - \lvert I^*_{\beta,\delta}\lvert\\
&>0.
\end{align*}
Moreover, since $\bar{y}\in \Delta_{x,\tilde{x}}$, we have $\varphi_{x}(s_{\bar{y}},t)=\varphi_{\tilde{x}}(s_{\bar{y}},t)$, and we thus can prove the same inequalities with $\tilde{r}_1$ and $\tilde{r}_2$ instead of $r_1$ and $r_2$.\\
Finally we thus get that the interval $(R_1,R_2)=(r_1,r_2)\cap (\tilde{r_1},\tilde{r_2})$ is well defined and
\begin{equation*}
R_2-R_1\geq 2\left( h\varepsilon - \lvert I^*_{\beta,\delta} \lvert \right).
\end{equation*}
\end{proof}

\begin{Rema}
The fact that $\lvert \mathcal{J} \lvert = \pi$ is here to ensure that the process can go from $x$ and $\tilde{x}$ to $\bar{y}$ in the proof of Proposition \ref{Prop continuite de (X,T) dans convexe}. If $ \lvert \mathcal{J} \lvert <\pi$, since $x$ and $\tilde{x}$ are unspecified and $\bar{y}$ can therefore be everywhere on $\partial K$, nothing ensures that this path is available.
\end{Rema}

We can now state the following theorem on the speed of convergence of the stochastic billiard in the convex set $K$.

\begin{Theo}
Let $K\subset \RR^2$ satisfying Assumption $(\mathcal{K})$ with diameter $D$. 
Let $(X_t,V_t)_{t\geq 0}$ the stochastic billiard process evolving in $K$ and verifying Assumption $(\mathcal{H})$ with $\lvert \mathcal{J}\lvert=\pi$.\\
There exists a unique invariant probability measure $\chi$ on $K\times \SS$ for the process $(X_t,V_t)_{t\geq 0}$.\\ 
Moreover, let define $n_0$ and $p$ by \eqref{def n0 proc continu dans le convexe} and \eqref{def p proc continu dans le convexe} with $\zeta\in\left(0,\frac{2}{C}\right)$. Let consider $\eta$, $I^*_{\beta,\delta}$, $R_1,R_2$ as in Proposition \ref{Prop continuite de (X,T) dans convexe} and Lemma \ref{Lemme minoration phi'(s)}, and let define $\kappa$ by \eqref{def kappa proc continu dans le convexe}. \\
For all $t\geq 0$ and all $\lambda<\lambda_M$:
\begin{equation*}
\lVert \PP\left( X_t\in \cdot, V_t\in \cdot \right) - \chi \lVert_{TV} \leq C_{\lambda}\mathrm{e}^{-\lambda t},
\end{equation*}
where
\begin{equation*}
\lambda_M=\min\left\{\frac{1}{2D}\log\left( \frac{1}{1-p}\right) ; \frac{1}{2D}\log\left( \frac{-(1-p)+\sqrt{(1-p)^2+4p(1-\kappa)}}{2p(1-\kappa)}\right) \right\}
\end{equation*}
and
\begin{equation*}
C_\lambda=\frac{p\kappa \mathrm{e}^{5\lambda D}}{1-\mathrm{e}^{2\lambda D}(1-p)-\mathrm{e}^{4\lambda D}p(1-\kappa)}.
\end{equation*}
\end{Theo}

\begin{proof}
As previously, the existence of an invariant probability measure for the stochastic billiard process comes from the compactness of $K\times\SS^1$. The following proof ensures its uniqueness and gives an explicit speed of convergence.\\
Let $(X_t,V_t)_{t\geq 0}$ and $(\tilde{X}_t,\tilde{V}_t)_{t\geq 0}$ be two versions of the stochastic billiard with $(X_0,V_0)=(x_0,v_0)\in K\times\S$ and $(\tilde{X}_0,\tilde{V}_0)=(\tilde{x}_0,\tilde{v}_0)\in K\times\S$.\\
We define (or recall the definition for $T_0$ and $\tilde{T}_0$):
\[ T_0=\inf\{ t\geq 0, x_0+tv_0 \notin K\}, ~~~~~ w=x_0+T_0 v_0\in\partial K,  \]
and
\[\tilde{T}_0=\inf\{ t\geq 0, \tilde{x}_0+t\tilde{v}_0 \notin K\}, ~~~~~ \tilde{w}=\tilde{x}_0+\tilde{T}_0 \tilde{v}_0 \in \partial K .\]
\textbf{Step 1}. From Proposition \ref{Prop T1 dans le convexe}, we deduce that for all $x\in\partial K$ and all $\zeta\in \left( 0,\frac{1}{C}\right)$, $T_n^x$ is $(c\rho_{\min})^n\zeta^{n-1}$-continuous on the interval $\Gamma_n=\left[ (n-1)\zeta, \frac{nC}{2}-(n-1)\zeta\right]$.\\
Let thus $\zeta \in \left( 0,\frac{1}{C}\right)$ and let define 
\begin{equation}\label{def n0 proc continu dans le convexe}
n_0= \min\left\{ n\geq 1: \lvert \Gamma_n\lvert >D\right\} = \left\lfloor \frac{D-2\zeta}{2\left(\frac{1}{C}-1\right)} \right\rfloor +1.
\end{equation}
The variables $T_0+T_{n_0}^w$ and $\tilde{T}_0+\tilde{T}_{n_0}^{\tilde{w}}$ are both $(c\rho_{\min})^{n_0}\zeta^{n_0-1}$-continuous on \\$\left[ T_0+(n_0-1)\zeta, T_0+\frac{nC}{2}-(n_0-1)\zeta\right] \cap \left[ \tilde{T}_0+(n_0-1)\zeta, \tilde{T}_0+\frac{nC}{2}-(n_0-1)\zeta\right]$. Since $\left\lvert T_0 - \tilde{T}_0\right\lvert\leq D$, this intersection is non-empty and its length is larger that $\frac{2n_0}{C}-2(n_0-1)\zeta-D$.\\
Let define 
\begin{equation}\label{def p proc continu dans le convexe}
p=(c\rho_{\min})^{n_0}\zeta^{n_0-1}\left(\frac{2n_0}{C}-2(n_0-1)\zeta-D \right).
\end{equation}
Using the fact that the for all $w\in\partial K$, $T^w_{n_0}\leq n_0D$ almost surely, we deduce that we can construct a coupling such that the coupling-time $T^1_c$ of $T_0+T_{n_0}^w$ and $\tilde{T}_0+\tilde{T}_{n_0}^{\tilde{w}}$ satisfies:
\begin{equation*}
T_c^1 \leq_{st} T_0 + n_0DG^1
\end{equation*}
with $G^1\sim \mathcal{G}\left(p \right)$.\medskip\\
\textbf{Step 2}. Once the coupling of these times has succeed, we try to couple the couples \\$\left(X_{T_2}^{X_{T^1_c}^w},T_2^{X_{T^c_1}^w}\right)$ and $\left(\tilde{X}_{\tilde{T}_2}^{\tilde{X}_{T^1_c}^{\tilde{w}}},\tilde{T}_2^{\tilde{X}_{T^1_c}^{\tilde{w}}} \right)$. By the Proposition \ref{Prop continuite de (X,T) dans convexe}, we can construct a coupling such that
\begin{equation*}
\PP\left( X_{T_2}^{X_{T^1_c}^w}=\tilde{X}_{\tilde{T}_2}^{\tilde{X}_{T_c^1}^{\tilde{w}}} ~~\text{and} ~~ T_2^{X_{T_c^1}^w}=\tilde{T}_2^{\tilde{X}_{T_c^1}^{\tilde{w}}} \right) \geq \eta\lvert I^*_{\beta,\delta}\lvert (R_2-R_1).
\end{equation*}
Defining 
\begin{equation}\label{def kappa proc continu dans le convexe}
\kappa=\eta\lvert I^*_{\beta,\delta}\lvert (R_2-R_1),
\end{equation}
we get that the entire coupling-time of the two processes satisfies:
\begin{equation*}
\hat{T}\leq_{st} T_0+\sum_{l=1}^G \left( n_0D G^l +n_0D \right)=T_0+\sum_{l=1}^G \left( n_0D (G^l +1) \right)
\end{equation*}
where $G$ as a geometric distribution with parameter $\kappa$ and the $(G^l)_{l\geq 1}$ are independent geometric random variables with parameter $p$, and independent of $G$.\\
Finally, we get
\begin{align*}
\PP\left( \hat{T}>t\right) &\leq \mathrm{e}^{-\lambda t} \frac{p\kappa \mathrm{e}^{5\lambda D}}{1-\mathrm{e}^{2\lambda D}(1-p)-\mathrm{e}^{4\lambda D}p(1-\kappa)},
\end{align*}
for all $\lambda \in \left(0,\lambda_M \right)$.

\end{proof}

\section{Discussion}

All the results presented in this paper are in dimension $2$. However, the ideas developed here can be adapted to higher dimensions. Let us briefly explain it.
\subsubsection*{Stochastic billiard in a ball of $\RR^d$}
Let us first look at the stochastic billiard $(X,V)$ in a ball $\mathcal{B}\subset\RR^d$ with $d\geq 2$.\\
As we have done in Section \ref{Section Stochastic billiard in the disc}, we can represent the Markov chain $(X_{T_n},V_{T_n})_{n\geq 0}$ by another Markov chain. Indeed, for $n\geq 1$, the position $X_{T_n}\in\partial \mathcal{B}$ can be uniquely represented by its hyperspherical coordinates: a $(d-1)$-tuple $(\Phi_n^1,\cdots,\Phi_n^{d-1})$ with $\Phi_n^1,\cdots,\Phi_n^{d-2}\in \left[ 0,\pi\right)$ and $\Phi_n^{d-1}\in \left[0,2\pi\right)$.\\
Similarly, for $n\geq 1$, the vector speed $V_{T_n}\in \left\{v\in\SS^{d-1}: v\cdot n_{X_{T_n}}\geq 0\right\}$ can be represented by its hyperspherical coordinates.\\
Thereby, we can give relations between the different random variables as in Proposition \ref{Prop lien entre les variables dans le cercle}, and in theory, we can do explicit computations to get lower bounds on the needed density function. Then the same coupling method in two steps can be applied. Nevertheless, it could be difficult to manage the computations in practice when the dimension increases.

\subsubsection*{Stochastic billiard in a convex set $K\subset\RR^d$}
To get bounds on the speed of convergence of the stochastic billiard $(X,V)$ in a convex set $K\subset \RR^d$, $d\geq 2$, satisfying Assumption $(\mathcal{K})$, we can apply exactly the same method as in Section \ref{Section Stochastic billiard in a convex set with bounded curvature}. The main difficulty could be the proof of the equivalent of Proposition \ref{Prop continuite de (X,T) dans convexe}. But it can easily be adapted, and we refer to the proof of Lemma 5.1 in \cite{CPSV}, where the authors lead the proof in dimension $d\geq 3$.

\paragraph{Acknowledgements.} The author thanks H\'el\`ene Gu\'erin and Florent Malrieu for their help in this work.\\This work was supported by the Agence Nationale de la Recherche project PIECE 12-JS01-0006-01.

\bibliographystyle{plain}
\bibliography{bibliobilliard}

\end{document}